\documentclass{amsart}
\usepackage{amsmath}
\usepackage{amssymb}
\usepackage{amsthm}
\usepackage{enumerate}
\theoremstyle{definition}
\newtheorem{definition}{Definition}[section]
\theoremstyle{plain}
\newtheorem{lemma}[definition]{Lemma}
\newtheorem{theorem}[definition]{Theorem}
\newtheorem{proposition}[definition]{Proposition}
\newtheorem{corollary}[definition]{Corollary}
\theoremstyle{remark}

\newtheorem{example}[definition]{Example}

\makeatletter
\@namedef{subjclassname@2020}{
  \textup{2020} Mathematics Subject Classification}
\makeatother

\newcommand{\mycl}{\operatorname{cl}}
\newcommand{\myrk}{\operatorname{rk}}
\newcommand{\myIso}{\operatorname{Iso}}
\newcommand{\myLpt}{\operatorname{lpt}}
\newcommand{\myDef}{\operatorname{Def}}
\newcommand{\mydim}{\operatorname{Dim}}
\newcommand{\mydims}{\dim_{\text{w}}}
\newcommand{\mySet}{\mathbf{DIM}}
\newcommand{\myNum}{\mathfrak n_{\dim}}
\newcommand{\mySelf}{\mathfrak D}
\newcommand{\myfineq}{\sim_{\text{fin}}}
\newcommand{\topdim}{\dim_{\text{top}}}

\begin{document}
\title[Dimension functions in ordered structures]{Cardinality of the sets of dimension functions in ordered structures}
\author[M. Fujita]{Masato Fujita}
\address{Department of Liberal Arts,
Japan Coast Guard Academy,
5-1 Wakaba-cho, Kure, Hiroshima 737-8512, Japan}
\email{fujita.masato.p34@kyoto-u.jp}
\thanks{The author was supported by JSPS KAKENHI Grant Number JP25K07109.}

\begin{abstract}
	We compute the cardinality $\myNum(\mathcal M)$ of the sets of dimension functions on the ordered structures $\mathcal M$.
	The inequality $\myNum(\mathcal M) \leq 1$ holds if $\mathcal M$ is a d-minimal expansion of an ordered group.
	If $\mathcal M$ is o-minimal and $\myNum(\mathcal M)<\infty$, there exists a positive integer $m$ such that $\myNum(\mathcal M)=2^m-1$.
	For every positive integer $m$, there exists a weakly o-minimal expansion $\mathcal M$ of an ordered divisible Abelian group such that $\myNum(\mathcal M)=m$.
\end{abstract}

\subjclass[2020]{Primary 03C64; Secondary 05C45, 54F45}

\keywords{o-minimality; d-minimality; weak o-minimality; dimension}

\maketitle

\section{Introduction}\label{sec:intro}
In this paper, we compute the cardinality of the sets of dimension functions on the ordered structures $\mathcal M$.
First we clarify the definition of dimension functions.

Let $\mathcal M=(M,\ldots)$ be a structure.
For $d<n$, let $\Pi_{d}^n:M^n \to M^{n-d}$ be the coordinate projection of $M^n$ forgetting the last $d$ coordinates.
We abbreviate $\Pi_1^n$ by $\Pi$ when $n$ is clear from the context.
For every $x \in \Pi_d^n(X)$, the fiber $\{y \in M^d\;|\;(x,y) \in X\}$ is denoted by $X_x^{\Pi^n_d}$.
We write $X_x^{\Pi}$ instead of $X_x^{\Pi_1^n}$.
The following is the definition of dimension function employed in this paper.
\begin{definition}[{\cite{vdD2}}]\label{def:dimension}
	Let $\mathcal M=(M,\ldots)$ be a structure.
	Let $\myDef_n(\mathcal M)$ be the collection of all subsets of $M^n$ definable in $\mathcal M$.
	Put $\myDef(\mathcal M):=\bigcup_{n>0}\myDef_n(\mathcal M)$.
	A function $\dim: \myDef(\mathcal M) \to \mathbb N \cup \{-\infty\}$, where $\mathbb N$ is the set of nonnegative integers, is called a \textit{dimension function on $\mathcal M$} if the following conditions are satisfied for every $n >0$:
	\begin{enumerate}
		\item[(1)] $\dim(\emptyset)=-\infty$, $\dim(\{a\})=0$ for each $a \in M$, $\dim(M^1)=1$.
		\item[$(2)_n$] Let $X$ and $Y$ be in $\myDef_n(\mathcal M)$.
		The equality
		\begin{align*}
			\dim(X \cup Y)=\max\{\dim(X),\dim(Y)\}
		\end{align*}
		holds.
		In particular, we have $\dim X \leq \dim Y$ if $X \subseteq Y$.
		\item[$(3)_n$] $\dim (X)=\dim(X^\sigma)$ for each $X \in \myDef_n(\mathcal M)$ and each permutation $\sigma$ of $\{1,\ldots, n\}$, where $$X^\sigma:=\{(x_{\sigma(1)},\ldots,x_{\sigma(n)}) \in M^n\;|\; (x_1,\ldots,x_n) \in X\}.$$
		\item[$(4)_n$] Suppose $n>1$. Let $X \in \myDef_n(\mathcal M)$. Put $$X(i):=\{x \in \Pi(X)\;|\; \dim X_x^{\Pi} = i\}$$ for $i=0,1$.
		Then, $X(i) \in \myDef_{n-1}(\mathcal M)$ and 
		\begin{equation}
			\dim(X \cap \Pi^{-1}(X(i)))=\dim(X(i))+i  \tag{*} \label{eq:4n}
		\end{equation}
	\end{enumerate}
	
	For the technical reason, we introduce additional notations.
	$M^0$ represent a singleton and we put $\dim M^0=0$.
	For every $n >0$, a (coordinate) projection of $M^n$ onto $M^0$ means the trivial map. 
	It is obvious that $\dim X \leq n$ for every definable subset $X$ of $M^n$.
	
\end{definition}
Condition $(4)_n$ in Definition \ref{def:dimension} is called the addition property.
Definition \ref{def:dimension} is the list of properties which every `good dimension theory' should enjoy.
The set of dimension functions on $\mathcal M$ is denoted by $\mySet(\mathcal M)$ and its cardinality is denoted by $\myNum(\mathcal M)$.

In this paper, we mainly discuss about dimension functions under the assumption that the given structure is an expansion of a dense linear order without endpoints (DLO).
Before we move on to our main topic, we give brief comments on dimension functions on structures not defining orders. 
For instance, the Morley rank is a unique dimension function on strongly minimal structures (Proposition \ref{prop:strongly}).
In \cite{vdD2}, van den Dries gave another example of dimension function related to Henselian fields.

From now on, let us consider structures defining orders.
An order induces the order topology over the universe of the structure.
A standard candidate of dimension function is the topological dimension given below.
\begin{example}\label{ex:top_dim}
	Let $\mathcal M=(M,<,\ldots)$ be an expansion of DLO. 
	$M$ is equipped with the order topology induced from the order relation $<$, and $M^n$ is equipped with the product topology.
	$M^0$ is defined as the singleton with the trivial topology.
	For each definable subset $X$ of $M^n$, we define $\topdim X$ as follows:
	If $X=\emptyset$, $\topdim X=-\infty$.
	If $X$ is not empty, $\topdim X$ is the maximum of nonnegative integer $d$ such that there exists a coordinate projection $\pi:M^n \to M^d$ satisfying that $\pi(X)$ has a nonempty interior.
\end{example}
In fact, the topological dimension $\topdim$ is a dimension function in the sense of Definition \ref{def:dimension} in several important cases such as the case where $\mathcal M$ is o-minimal \cite{vdD, Mathews}, definably complete locally o-minimal \cite{Fuji3,FKK} and Baire constructible structures \cite[Corollary 3.16]{Fuji4} including d-minimal expansions of an ordered field.

Questions tackled in this paper are whether other dimension functions exist and, if that's the case, how many dimension functions exist.
We compute $\myNum(\mathcal M)$ for several ordered structures $\mathcal M$ as follows: 
\begin{enumerate}
	\item[(a)] $\myNum((\mathbb N,<))=\myNum((\mathbb Z,<))=0$ (Corollary \ref{cor:non});
	\item[(b)] $\myNum(\mathcal M) \leq 1$ if $\mathcal M$ is either a d-minimal expansion of an ordered group or a weakly o-minimal expansion of an ordered field (Theorem \ref{thm:unique});
	\item[(c)] $\myNum(\mathcal M)=|M|$ if $\mathcal M=(M,<)$ is a model of DLO (Proposition \ref{prop:DLO}), where $|M|$ denotes the cardinality of the set $M$;
	\item[(d)] For every o-minimal structure $\mathcal M$, if $\myNum(\mathcal M)<\infty$, there exists a positive integer $m$ such that $\myNum(\mathcal M)=2^m-1$ (Theorem \ref{thm:omin_finite}).	
	For every integer $m>0$, there exists an o-minimal theory $T$ such that $\myNum(\mathcal M)=2^{m+1}-1$ for every model $\mathcal M$ of $T$ (Proposition \ref{prop:omin_finite_example}).
	\item[(e)] For every integer $m>1$, there exists a weakly o-minimal theory $T$ extending the theory of ordered divisible Abelian groups such that $\myNum(\mathcal M)=m$ for every model $\mathcal M$ of $T$ (Theorem \ref{thm:weakly}).  
\end{enumerate}

This paper is organized as follows:
In Section \ref{sec:prelim}, we prove that functions satisfying some weaker conditions than those in Definition \ref{def:dimension} is a dimension function.
This result enables us to prove that given functions are dimension functions with less efforts. 
In Section \ref{sec:uniq}, we prove facts (a-e) introduced above.

Throughout, `definable' means `definable in the structure with parameters.'
Consider a linearly ordered set without endpoints $(M,<)$.
An open interval is a nonempty set of the form $\{x \in M\;|\; a < x < b\}$ for some $a,b \in M \cup \{\pm \infty\}$.
It is denoted by $(a,b)$ in this paper.
We use the same notation for the pair of elements.
Readers will not be confused due to this abuse of notations.
An open box is the Cartesian product of open intervals.
The closed interval is defined similarly and denoted by $[a,b]$.
We define $[a,b)$ similarly.
When an expansion $\mathcal M=(M,<,\ldots)$ of a dense linear order without endpoints is given, the set $M$ is equipped with the order topology induced from the order $<$. 
The Cartesian product $M^n$ is equipped with the product topology of the order topology.

\section{Weak dimension function is a dimension function}\label{sec:prelim}

First we recall several basic definitions for the later use.
Let $\mathcal M=(M,<)$ be an expansion of DLO.
$\mathcal M$ is \textit{o-minimal} if every definable subset of $M$ is a union of a finite set and finitely many open intervals \cite{vdD}.
$\mathcal M$ is \textit{definably complete} if, for every definable subset $X$ of $M$, the supremum and infimum of $X$ exist in $M \cup \{\pm \infty\}$ \cite{M}.
$\mathcal M$ is \textit{weakly o-minimal} if every definable subset of $M$ is a union of finitely many convex sets \cite{MMS}.
$\mathcal M$ is \textit{locally o-minimal} if, for every $a \in M$ and every definable subset $X$ of $M$, there exists an open interval such that $X \cap I$ is a union of a finite set and finitely many open intervals \cite{TV}.
Finally, $\mathcal M$ is \textit{d-minimal} if it is definably complete and, for every structure $\mathcal N=(N,<,\ldots)$ elementary equivalent to $\mathcal M$, every definable subset $X$ of $N$ is a union of an open set and finitely many discrete sets \cite{Fornasiero,M2}.
The following diagram illustrates the implications among these concepts:
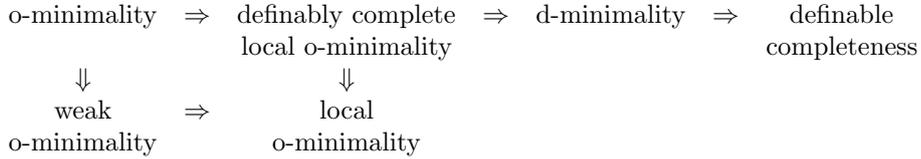
\begin{center}
\begin{tabular}{ccccccc}
o-minimality & $\Rightarrow$  & definably complete & $\Rightarrow$ & d-minimality & $\Rightarrow$ & definable \\
 &  & local o-minimality & & & &  completeness\\
$\Downarrow$ & & $\Downarrow$ &&&&\\
weak & $\Rightarrow$ & local &&&&\\
o-minimality & & o-minimality &&&&
\end{tabular}
\end{center}
These implications are strict, that is, the converse does not hold.

Next we introduce several facts used many times in this paper.
\begin{proposition}[{\cite[1.3, Corollary 1.5]{vdD2}}]\label{prop:dim}
	Let $\mathcal M=(M,\ldots)$ be a structure and $\dim: \myDef(\mathcal M) \to \mathbb N \cup \{-\infty\}$ be a dimension function on $\mathcal M$.
	Let $X \in \myDef_m(\mathcal M)$ and $f:X \to M^n$ be a map whose graph belongs to $\myDef_{m+n}(\mathcal M)$.
	Then, we have:
	\begin{enumerate}
		\item[(1)] The function $\dim$ is completely determined by its restriction to $\myDef_1(\mathcal M)$;
		\item[(2)] $\dim f(X) \leq \dim X$;
		In particular, $\dim X=\dim f(X)$ if $f$ is injective;
		\item[(3)] For every $0 \leq i \leq m$, $B(i):=\{y \in M^n\;|\; \dim f^{-1}(y)=i\}$ belongs to $\myDef_m(\mathcal M)$ and $\dim(f^{-1}(B(i)))=\dim B(i)+i$.
	\end{enumerate}
\end{proposition}

We introduce a weaker notion than that in Definition \ref{def:dimension}, which turns out to be equivalent to that in Definition \ref{def:dimension}.
We show the equivalence in Proposition \ref{prop:weak}.
\begin{definition}\label{def:dimension2}
	Let $\mathcal M=(M,\ldots)$ be a structure.
	A  function $\dim: \myDef(\mathcal M) \to \mathbb N \cup \{-\infty\}$ on $\mathcal M$ is a \textit{weak dimension function} on $\mathcal M$ if conditions (1), $(2)_1$ in Definition \ref{def:dimension} and the following condition (3'), which is equivalent to $(3)_2$ in Definition \ref{def:dimension}, and $(4')_n$ for $n>1$ is satisfied:
	\begin{enumerate}
		\item[$(3')$] $\dim (X)=\dim(X^{\text{switch}})$ for each $X \in \myDef_2(\mathcal M)$, where $$X^{\text{switch}}:=\{(x_1,x_2) \in M^2\;|\; (x_2,x_1) \in X\}.$$
		\item[$(4')_n$] Suppose $n>1$. Let $X \in \myDef_n(\mathcal M)$. Put $$X(i):=\{x \in \Pi(X)\;|\; \dim X_x^{\Pi} = i\}$$ for $i=0,1$.
		Then, $X(i) \in \myDef_{n-1}(\mathcal M)$ and 
		\begin{equation}
			\dim(X)=\max\{\dim(X(i))+i\;|\; i=0,1\} \tag{**} \label{eq:4ndash}
		\end{equation}
	\end{enumerate}
	Observe $(4')_n$ implies $(4)_n$.
	In fact, if we apply equality (\ref{eq:4ndash}) to $X \cap \Pi^{-1}(X(i))$ for $i=0,1$, we get equality (\ref{eq:4n}) in Definition \ref{def:dimension}$(4)_n$.
\end{definition}

Now we start to prove that a weak dimension function is a dimension function.

\begin{lemma}\label{lem:2n}
	Let $\mathcal M=(M,\ldots)$ be a structure.
	Let $\dim: \myDef(\mathcal M) \to \mathbb N \cup \{-\infty\}$ be a function.
	Let $n$ be a positive integer with $n>1$.
	If $(2)_1$, $(2)_{n-1}$ in Definition \ref{def:dimension} and $(4')_n$ in Definition \ref{def:dimension2} hold for $\dim$, then $(2)_n$ holds for $\dim$.
\end{lemma}
\begin{proof}
	For every definable subset $X$ of $M^n$, we put $$S_i(X):=\{x \in M^{n-1}\;|\;\dim X_x^{\Pi} =i\}$$ for $i=0,1$.
	By $(2)_1$, we have 
	\begin{align*}
		& S_1(X \cup Y)=S_1(X ) \cup S_1(Y)\text{ and }\\
		& S_0(X \cup Y)=(S_0(X) \setminus S_1(Y)) \cup (S_0(Y) \setminus S_1(X)).
	\end{align*}
	By symmetry, we may suppose $\dim X \geq \dim Y$.
	
	Suppose $\dim S_0(X)>\dim S_1(X)$.
	We have $\dim(X)=\dim(S_0(X))$ by $(4')_n$.
	We want to show $$\dim (S_0(X) \setminus S_1(Y)) = \dim X.$$
	Assume for contradiction $\dim (S_0(X) \setminus S_1(Y)) \neq \dim X$.
	Observe that 
	\begin{align*}
		\dim \Pi(X) &= \max \{\dim S_1(X), \dim S_0(X)\} \leq  \max\{\dim S_1(X)+1,\dim S_0(X)\}\\
		&= \dim X
	\end{align*}
	by $(2)_{n-1}$ and $(4')_n$.
	By $(2)_{n-1}$ and this inequality, we have $\dim (S_0(X) \setminus S_1(Y)) < \dim X$.
	We have $$\dim X = \dim S_0(X) =\max \{\dim (S_0(X) \cap S_1(Y)) , \dim  (S_0(X) \setminus S_1(Y))\}$$
	by $(2)_{n-1}$. 
	We obtain $\dim (S_0(X) \cap S_1(Y)) = \dim X$ because $\dim (S_0(X) \setminus S_1(Y)) < \dim X$.
	We get 
	\begin{align*}
		\dim Y > \dim S_1(Y) \geq \dim S_0(X) \cap S_1(Y) = \dim X
	\end{align*}
	by $(4')_n$ and $(2)_{n-1}$, which is absurd.
	We have shown $\dim (S_0(X) \setminus S_1(Y)) = \dim X$.
	
	On the other hand, we have 
	\begin{align*}
		\dim (S_0(Y) \setminus S_1(X)) \leq \dim S_0(Y) \leq \dim Y \leq \dim X
	\end{align*}
	by $(2)_{n-1}$, $(4')_n$ and the case hypothesis.
	We have
	\begin{align*}
		\dim S_0(X \cup Y)&=\max\{\dim(S_0(X) \setminus S_1(Y)),\dim(S_0(Y) \setminus S_1(X))\} =\dim X
	\end{align*}
	by $(2)_{n-1}$.
	
	We also have 
	\begin{align*}
		\dim S_1(X \cup Y) &=\max \{\dim S_1(X),\dim S_1(Y)\} \\
		&<\max \{\dim X, \dim Y\}=\dim X
	\end{align*}
	by $(2)_{n-1}$ and $(4')_n$.
	We get $$\dim (X \cup Y) = \max\{\dim S_0(X \cup Y),\dim S_1(X \cup Y)+1\}=\dim X.$$
	
	Suppose $\dim S_0(X) \leq \dim S_1(X)$.
	We have $\dim(X)=\dim(S_1(X))+1$ by $(4')_n$.
	We have 
	\begin{align*}
		&\dim (S_0(X) \setminus S_1(Y)) \leq \dim S_0(X) \leq \dim X
	\end{align*}
	by $(2)_{n-1}$ and $(4')_n$.
	We get $\dim (S_0(Y) \setminus S_1(X)) \leq \dim Y$ in the same manner.
	We obtain  $\dim (S_0(Y) \setminus S_1(X)) \leq \dim X$ by the case hypothesis.
	Therefore, we have $\dim S_0(X \cup Y) \leq \dim X$.
	It is obvious $\dim S_1(Y) < \dim Y \leq \dim X$ by $(4')_n$.
	We have $$\dim S_1(X \cup Y) = \max\{\dim S_1(X), \dim S_1(Y)\} = \dim X-1.$$
	Therefore, $\dim (X \cup Y) =\dim X$ by $(4')_n$. 
\end{proof}

\begin{lemma}\label{lem:5n}
	Let $\mathcal M=(M,\ldots)$ be a structure.
	Let $\dim: \myDef(\mathcal M) \to \mathbb N \cup \{-\infty\}$ be a function.
	Let $n$ be a positive integer.
	Consider the following property $(5)_n$.
	\begin{enumerate}
		\item[$(5)_n$] Suppose $n>1$. For $0<d<n$ and $0 \leq k \leq d$, put $$X(d,k):=\{x \in \Pi_d^n(X)\;|\; \mydim X_x^{\Pi_d^n} = k\}.$$
		Then, $X(d,k) \in \myDef_{n-d}(\mathcal M)$ and 
		\begin{equation*}
			\dim (X \cap (\Pi_d^n)^{-1}(X(d,k))) = \dim X(d,k)+k. 
		\end{equation*}
	\end{enumerate}
	If $(2)_i$ and $(4)_i$ in Definition \ref{def:dimension} hold for $\dim $ and $i \leq n$, then $(5)_n$ holds for $\dim$.
\end{lemma}
\begin{proof}
	We prove this lemma by induction on $n$ and $d$ under the lexicographic order.
	If $n=2$, $(5)_n$ follows from $(4)_2$.
	If $d=1$, $(5)_n$ follows from $(4)_n$.
	We suppose $d>1$.
	Let $X(i)$ be as in the statement of $(4)_n$ for $i=0,1$.
	Put $\mathcal I:=\{0 \leq j <d\} \cup \{-\infty\}$.
	Define $-\infty + k = k + (-\infty) =-\infty$ for $k \in \mathcal I$.
	Apply $(5)_{n-1}$ to $X(i)$, then $$U\langle i,j \rangle := \{t \in \Pi_d^n(X)\;|\; \dim (X(i))_t^{\Pi_{d-1}^{n-1}} =j \}$$ is definable for $i=0,1$ and $j \in \mathcal I$.
	Apply $(4)_d$ to $(X \cap \Pi^{-1}(X(i)))_x^{\Pi_d^n}$, then we have $$\dim (X \cap \Pi^{-1}(X(i)))_x^{\Pi_d^n}=i+j$$ for every $x \in U\langle i,j \rangle$ by the definition of $X(i)$.
	Put $V\langle j_0,j_1 \rangle:=U\langle 0,j_0 \rangle \cap U\langle 1,j_1 \rangle$ for $j_0,j_1 \in \mathcal I$.
	$V\langle j_0,j_1 \rangle$ is definable.
	For $x \in V\langle j_0,j_1 \rangle$, we have 
	\begin{align}
		\dim (X_x^{\Pi_d^n})&=\max\{\dim (X \cap \Pi^{-1}(X(i)))_x^{\Pi_d^n}\;|\; i=0,1 \}
		=\max\{j_0,j_1+1\}\tag{a}\label{eq:c}
	\end{align}
	by $(2)_d$.
	The previous equality implies
	\begin{equation}
		X(d,k)=\bigcup_{(j_0,j_1) \in \mathcal I \times \mathcal I, \max\{j_0,j_1+1\}=k} V\langle j_0,j_1 \rangle. \tag{b} \label{eq:d}
	\end{equation}
	$X(d,k)$ is definable because $V\langle j_0,j_1 \rangle$ are so.
	
	Apply $(5)_{n-1}$ to $X(i) \cap (\Pi_{d-1}^{n-1})^{-1}(V\langle j_0,j_1 \rangle)$, then we have
	\begin{align*}
		\dim (X(i) \cap (\Pi_{d-1}^{n-1})^{-1}(V\langle j_0,j_1 \rangle)) = \dim V\langle j_0,j_1 \rangle + j_i
	\end{align*}
	for $i=0,1$.
	By $(4)_n$, we get
	\begin{align*}
		\dim (X \cap \Pi^{-1}(X(i) \cap (\Pi_{d-1}^{n-1})^{-1}(V\langle j_0,j_1 \rangle)))
		= \dim V\langle j_0,j_1 \rangle + j_i + i.
	\end{align*}
	Since $\Pi(X)=X(0) \cup X(1)$, we have 
	\begin{align*}
		&\bigcup_{i=0}^1 X \cap \Pi^{-1}(X(i) \cap (\Pi_{d-1}^{n-1})^{-1}(V\langle j_0,j_1 \rangle))\\
		&\qquad = X \cap \Pi^{-1}(\Pi(X) \cap (\Pi_{d-1}^{n-1})^{-1}(V\langle j_0,j_1 \rangle))\\
		&\qquad = X \cap (\Pi_d^n)^{-1}(V\langle j_0,j_1 \rangle).
	\end{align*}
	Therefore, by $(2)_n$, we get
	\begin{align*}
		\dim(X \cap (\Pi_d^n)^{-1}(V\langle j_0,j_1 \rangle)) = \dim V\langle j_0,j_1 \rangle + \max\{j_0,j_1 + 1\}. \tag{c} \label{eq:e}
	\end{align*}
	Using $(2)_{n-d}$, $(2)_n$ and equations (\ref{eq:c}-\ref{eq:e}), we obtain the equation in $(5)_n$.
\end{proof}

Proposition \ref{prop:weak} makes it easier to prove that given functions are dimension functions. 

\begin{proposition}\label{prop:weak}
	Every weak dimension function on a structure $\mathcal M$ is a dimension function. 
\end{proposition}
\begin{proof}
	Note that $(2)_n$ holds for $n \geq 1$ by Lemma \ref{lem:2n}.
	Let $(5)_n$ be the property described in Lemma \ref{lem:5n}.
	In our setting, property $(5)_n$ holds for every $n>1$.
	By induction on $n$, we prove property $(3)_n$ in Definition \ref{def:dimension} holds if $\dim$ satisfies the conditions in Definition \ref{def:dimension2}.
	
	If $n=2$, property $(3)_n$ in Definition \ref{def:dimension} immediately follows from property $(3')$ in Definition \ref{def:dimension2}.
	
	Suppose $n>2$.
	Every permutation is a product of transpositions of the form $(l-1,l)$.
	Therefore, we may assume that $\sigma$ is the transposition $(l-1,l)$ without loss of generality.
	
	Suppose $l<n$.
	Put $d:=n-l$.
	We use the notations in $(5)_n$ of Lemma \ref{lem:5n}.
	Let $X^{\sigma}(d,k):=\{x \in \Pi_d^n(X^{\sigma})\;|\; \mydim (X^{\sigma})_x^{\Pi_d^n} = k\}$.
	Put $$(X(d,k))^{\sigma}:=\{(x_1,\ldots,x_{l-2},x_l,x_{l-1}) \in M^l\;|\; (x_1,\ldots,x_l) \in X(d,k)\}.$$
	We obviously have $X^{\sigma}(d,k)=(X(d,k))^{\sigma}$.
	By the induction hypothesis, we have $\dim X^{\sigma}(d,k) = \dim (X(d,k))^{\sigma} =  \dim X(d,k)$.
	By $(2)_n$ and $(5)_n$, we have
	\begin{align*}
		\dim X &= \max\{\dim (X \cap (\Pi_d^n)^{-1}(X(d,k)))\;|\; 0 \leq k \leq d\}\\
		&= \max\{\dim X(d,k) + k\;|\; 0 \leq k \leq d\}\\
		&= \max\{\dim X^{\sigma}(d,k) + k\;|\; 0 \leq k \leq d\}\\
		&= \max\{\dim (X^{\sigma} \cap (\Pi_d^n)^{-1}(X^{\sigma}(d,k)))\;|\; 0 \leq k \leq d\}\\
		&=\dim X^{\sigma}.
	\end{align*}
	
	Suppose $l=n$.
	Apply property $(3')$ in Definition \ref{def:dimension2} to $X_x^{\Pi_2^n} \subseteq M^2$ for  $x \in \Pi_2^n(X)$.
	Then, we get $\dim(X_x^{\Pi_2^n})= \dim ((X^{\sigma})_x^{\Pi_2^n})$ for every $x \in \Pi_2^n(X)$.
	Therefore, we have $X^{\sigma}(2,k)=X(2,k)$ for $0 \leq k \leq 2$.
	Using $(2)_n$ and $(5)_n$, we can prove $\dim X=\dim X^{\sigma}$ in the same manner as in the previous case.
	We omit the details.
\end{proof}

\section{Dimension functions on structures}\label{sec:uniq}
In this section, we compute $\myNum(\mathcal M)$ for several structures.

\subsection{Non-ordered case}
First we introduce famous non-ordered structures $\mathcal M$ with $\myNum(\mathcal M)=1$. 
Suppose $\mathcal L$ is countable.
Recall that a structure $\mathcal N=(N,\ldots)$ is \textit{minimal} if, for every definable subset $X$ of $N$, either $X$ or $N \setminus X$ is finite.
Let $\mathcal M =(M,<)$ be a \textit{strongly minimal} structure, that is, it is a minimal structure such that every structure elementarily equivalent to it is also minimal. 
$(M,\operatorname{acl})$ is a pregeometry by \cite[Lemma B.1.5]{Zilber}, where $\operatorname{acl}$ denotes the algebraic closure operator.
Suppose $\myrk^{\operatorname{acl}}(M/\emptyset)$ is infinite.
The Morley rank of a subset $X$ of $M^n$ defined over a finite subset $A$ of $M$ is defined by $\myrk_{\text{Morley}}(X) :=\myrk_{\mathcal M}^{\operatorname{acl}}(X/A)$.

\begin{proposition}\label{prop:strongly}
	The Morley rank is a unique dimension function. 
\end{proposition}
\begin{proof}
	The Morley rank is a dimension function by \cite[Lemma B.1.26]{Zilber}.
	The uniqueness easily follows from Definition \ref{def:dimension} and Proposition \ref{prop:dim}(1) because every univariate definable subset is either finite or cofinite.
\end{proof}

\subsection{Upper bound of number of dimension functions for ordered structures}
Our next target is Theorem \ref{thm:unique}, which asserts $\myNum(\mathcal M) \leq 1$ for some ordered structures.
We prove several lemmas.
\begin{lemma}\label{lem:unbounded0}
	Suppose $\mathcal M=(M,<,\ldots)$ is an expansion of a linear order.
	Let $\dim: \myDef(\mathcal M) \to \mathbb N \cup \{-\infty\}$ be a dimension function.
	Let $X$ be a definable subset of $M$ such that at least one of the following conditions is satisfied:
	\begin{enumerate}
		\item[(i)] For every $a \in X$, $\dim (\{x \in X\;|\;x<a\}) \leq 0$;
		\item[(ii)] For every $a \in X$, $\dim (\{x \in X\;|\;x>a\}) \leq 0$;
	\end{enumerate}
	Then, $\dim X=0$.
\end{lemma}
\begin{proof}
	We consider the case where condition (i) is satisfied only.
	We can prove the lemma similarly in the case where condition (ii) holds.
	Consider the following definable sets:
	\begin{align*}
		&\Delta:=\{(x,x)\;|\;x \in X\};\\
		&Y_1:=\{(x,y) \in X \times X\;|\; y<x\};\\
		&Y_2:=\{(x,y) \in X \times X\;|\; x<y\}.
	\end{align*}
	Observe that, for every $x \in X$ and the fiber $(Y_1)_x^{\Pi_1^2}:=\{y \in M\;|\; (x,y) \in Y_1\}$, $\dim (Y_1)_x^{\Pi_1^2} =0$ holds by condition (i).
	By Definition \ref{def:dimension}(4), we have $$\dim Y_1 = \dim X + \dim (Y_1)_x^{\Pi_1^2} = \dim X+0=\dim X.$$
	By Definition \ref{def:dimension}(1,4), we have $\dim \Delta = \dim X$.
	We have $\dim Y_2=\dim Y_1=\dim X$ by Definition \ref{def:dimension}(3) because $Y_2$ is obtained from $Y_1$ by permuting the coordinates.
	By Definition \ref{def:dimension}(2), we have $$\dim (X \times X)=\max \{\dim Y_1, \dim \Delta, \dim Y_2\}=\dim X$$ because $X \times X= Y_1 \cup \Delta \cup Y_2$.
	On the other hand, we get $\dim (X \times X) = 2\dim X$ by Definition \ref{def:dimension}(4).
	These imply $\dim X=0$.
\end{proof}

\begin{corollary}\label{cor:unbounded0}
	Suppose $\mathcal M=(M,<,\ldots)$ is an expansion of a linear order.
	A dimension function $\dim: \myDef(\mathcal M) \to \mathbb N \cup \{-\infty\}$ is completely determined by the restriction to the collection $\myDef_{1,\text{bdd}}(\mathcal M)$ of bounded definable subsets of $M$.
	In addition, there exists a definable bounded subset $X$ of $M$ with $\dim (X)=1$.
\end{corollary}
\begin{proof}
	Let $X$ be an unbounded definable subset of $M$.
	If there exists a definable bounded subset $Y$ of $X$ such that $\dim Y=1$, then $\dim X=1$ by Definition \ref{def:dimension}(2).
	We have only to show that $\dim X=0$ if $\dim Y=0$ for every bounded definable subset $Y$ of $X$.
	
	First suppose $X$ is bounded below.
	Then the definable set $\{x \in X\;|\;x<a\}$ is bounded for every $a \in X$.
	We have $\dim (\{x \in X\;|\;x<a\})=0$ by the assumption.
	We have $\dim X=0$ by Lemma \ref{lem:unbounded0}.
	
	Next suppose $X$ is not bounded below.
	We have already proved that $\dim (\{x \in X\;|\;x>a\})=0$ for every $a \in X$ in the previous paragraph.
	We have $\dim X=0$ by Lemma \ref{lem:unbounded0}.
	
	Finally, we prove the `in addition' part.
	If $\dim (X)=0$ for every definable bounded definable subset $X$ of $M$, we have $\dim M=0$, which contradicts Definition \ref{def:dimension}(1).
\end{proof}

We give simple examples with $\myNum(\mathcal M)=0$.

\begin{corollary}\label{cor:non}
	$\myNum((\mathbb N,<))=\myNum((\mathbb Z,<))=0$.
\end{corollary} 
\begin{proof}
	Every unary definable bounded set is finite.
	If a dimension function $\dim$ on the given structure exists, the values of $\dim$ at the finite sets are zero by Definition \ref{def:dimension}(1,2).
	This contradicts the `in addition' part of Corollary \ref{cor:unbounded0}.
\end{proof}

\begin{lemma}\label{lem:discrete_closed}
	Suppose $\mathcal M=(M,<,\ldots)$ is a definably complete structure.
	Let $\dim:\myDef(\mathcal M) \to \mathbb N \cup \{-\infty\}$ be a dimension function.
	Then, $\dim X=0$ for every definable closed discrete subset $X$ of $M$.
\end{lemma}
\begin{proof}
	Let $X$ be a definable closed discrete subset of $M$.
	If $X$ is finite, $\dim X=0$ by Definition \ref{def:dimension}(1,2).
	We assume that $X$ is infinite.
	
	Take an arbitrary point $a \in X$.
	By Definition \ref{def:dimension}(1,2), if $\dim (X \cap (a,\infty))=0$ and $\dim(X \cap (-\infty,a))=0$, then $\dim X=0$.
	We only prove $\dim (X \cap (a,\infty))=0$ because the proof of the equality $\dim(X \cap (-\infty,a))=0$ is similar.
	
	By considering $X \cap (a,\infty)$ in place of $X$, we assume $X \subseteq (a,\infty)$ for the simplicity of notations.
	Let $$Y:=\{x \in X\;|\; \dim (X \cap (a,x]) \leq 0\},$$ which is a definable set by Definition \ref{def:dimension}(4).
	Let $c:=\inf X \in M$.
	Since $X$ is closed and discrete, we have $c \in X$.
	We have $c \in Y$ because $X \cap (0,c]=\{c\}$.
	Let $d:=\sup Y$.
	Since $Y$ is a subset of $X$, $Y$ is also discrete and closed.
	In particular, we have either $d \in Y$ or $d=\infty$.
	
	First suppose $d \in Y$.
	If $d=\sup X$, then we have $\dim X=0$ by the definition of $Y$.
	Assume for contradiction that $d<\sup X$.
	Then $Z:=\{x \in X\;|\; x>d\}$ is a nonempty definable subset of $X$.
	$Z$ is closed and discrete because $X$ is so.
	Therefore, we have $e:=\inf Z \in Z$.
	By the definition of $e$, we have $X \cap (a,e] = (X \cap (a,d]) \cup \{e\}$.
	We get $\dim (X \cap (a,e])=\max\{\dim (X \cap (a,d]) , \dim \{e\}\}=0$ by Definition \ref{def:dimension}(1,2).
	This means $e \in Y$, which contradicts the definitions of $d$ and $e$.
	
	Suppose $d=\infty$.
	Then, $X$ satisfies condition (i) of Lemma \ref{lem:unbounded0}.
	Lemma \ref{lem:unbounded0} implies $\dim X=0$.
\end{proof}

\begin{lemma}\label{lem:interval1}
	Suppose $\mathcal M=(M,<,+,0,\ldots)$ is a definably complete expansion of an ordered group.
	Let $\dim:\myDef(\mathcal M) \to \mathbb N \cup \{-\infty\}$ be a dimension function.
	We have $\dim X=1$ for every definable subset $X$ of $M$ having nonempty interior.
\end{lemma}
\begin{proof}
	Assume for contradiction that there exists a definable set $X$ such that $\dim X=0$ and it has a nonempty interior.
	There exists a nonempty bounded open interval $I=(a,b)$ contained in $X$.
	We have $0 \leq \dim I \leq \dim X=0$ by Definition \ref{def:dimension}(2).
	By shifting $I$, we may assume that $a=0$ by Proposition \ref{prop:dim}(2).
	We show $\dim ((0,\infty))=0$.
	
	Consider the sets 
	\begin{align*}
		&Y:=\{(x,y) \in M^2\;|\;0<y<x\}\text{ and }\\
		&Z:=\{x \in M\;|\; \dim Y_x^{\Pi_1^2}=0\}.
	\end{align*}
	We show $Z=(0,\infty)$.
	Since we have $\dim ((0,b)) = 0$, we have $b \in Z$.
	In particular, $Z$ is not empty.
	We have $\sup Z \in M \cup \{\infty\}$ because $\mathcal M$ is definably complete.
	Assume for contradiction that $c:=\sup Z \in M$.
	By Definition \ref{def:dimension}(2), we have $x \in Z$ whenever there exists $y \in Z$ such that $0<x<y$.
	Recall that $(M,+,0)$ is divisible Abelian group by \cite[Proposition 2.2]{M}.
	We have $3c/4 \in Z$ and $3c/2 \notin Z$.
	On the other hand, the interval $(3c/4,3c/2)$ is definably bijective to $(0,3c/4)$, and $\dim ((3c/4,3c/2)) = \dim ((0,3c/4))=0$ by Proposition \ref{prop:dim}(2).
	We have $$\dim ((0,3c/2))=\max\{\dim ((0,3c/4)), \dim(\{3c/4\}), \dim ((3c/4,3c/2))\}=0$$ by Definition \ref{def:dimension}(2).
	This means $3c/2 \in Z$, which contradicts the definition of $c$.
	We have shown $\sup Z=\infty$.
	This implies $Z=(0,\infty)$.
	
	For every $a \in (0,\infty)$, we have $\dim ((0,\infty) \cap (-\infty,a)) =0$ because $\sup Z=\infty$.
	We have $\dim (0,\infty) = 0$ by Lemma \ref{lem:unbounded0}.
	Since $M= (0,\infty) \cup \{0\} \cup (-\infty,0)$ and $(-\infty,0)$ is definably bijective to $(0,\infty)$, we have $\dim M=0$ by Definition \ref{def:dimension}(1,2) and Proposition \ref{prop:dim}(2), which contradicts the equality $\dim M=1$ in Definition \ref{def:dimension}(1).
\end{proof}

We recall the Cantor-Bendixson rank and its basic property.

\begin{definition}[\cite{FM}]\label{def:lpt}
	We denote the set of isolated points in $S$ by $\myIso(S)$ for any topological space $S$.
	We set $\myLpt(S):=S \setminus \myIso(S)$.
	
	Let $X$ be a nonempty closed subset of a topological space $S$.
	We set $X[ 0 ]=X$ and, for any $m>0$, we set $X [ m ] = \myLpt(X [ m-1 ])$.
	We say that $\myrk(X)=m$ if $X [ m ]=\emptyset$ and $X[ m-1 ] \neq \emptyset$.
	We say that $\myrk X = \infty$ when $X [ m ] \neq \emptyset$ for every natural number $m$.
\end{definition}

\begin{lemma}\label{lem:very_basic}
	Let $\mathcal M=(M,<)$ be an expansion of a dense linear order without endpoints.
	For a definable closed subset $A$ of $M$ with empty interior, $\myrk(A)=k$ if and only if $k$ is the least number of discrete sets whose union is $A$.
\end{lemma}
\begin{proof}
	See \cite[1.3]{FM}.
	The lemma is proven in \cite[1.3]{FM} when the underlying space is the set of reals.
	We can prove the lemma in the same manner as it.
	We omit the details.
\end{proof}

\begin{lemma}\label{lem:CBcase}
	Let $\mathcal M=(M,<,+,0,\ldots)$ and $\dim:\myDef(\mathcal M) \to \mathbb N \cup \{-\infty\}$ be as in Lemma \ref{lem:interval1}.
	Suppose that there exists a definable non-closed discrete subset $Z$ of $M$ such that the closure $\mycl(Z)$ of $Z$ is a union of finitely many definable discrete sets. 
	The equality $\dim X=0$ holds for every definable discrete subset $X$ of $M$.
\end{lemma}
\begin{proof}
	First, we construct a definable subset $P$ of $(0,\infty)$ such that 
	\begin{itemize}
		\item $P$ is discrete;
		\item $\mycl(P)=P \cup \{0\}$;
	\end{itemize}
%
	Put $Y_1=Z$ and $C_1:=\mycl(Y_1)$.
	If $C_1$ is discrete, then the subset $Y_1$ of $C_1$ is also discrete and closed, which is absurd.
	Therefore, $C_1$ is not discrete.
	This means $r:=\myrk(C_1) > 1$.
	Let $Y_k := \myIso(C_{k-1})$ and $C_k := C_{k-1} \setminus Y_k=\myLpt(C_{k-1})$ for $1<k \leq r$.
	By the definition of the Cantor-Bendixson rank, 
	\begin{itemize}
		\item $Y_{r-1}$ is not closed;
		\item $Y_{r}$ is closed;
		\item $\partial Y_{r-1} \subseteq Y_r$,
	\end{itemize}
	where $\partial Y_{r-1}$ is the frontier of $Y_{r-1}$.
	Put $Q:=Y_{r-1}$, then 
	\begin{itemize}
		\item $Q$ is discrete;
		\item $\partial Q$ is nonempty, discrete and closed.
	\end{itemize}
	By shifting $Q$, we may assume that $0 \in \partial Q$.
	Since $\partial Q$ is discrete and closed, we can choose $c>0$ such that $\partial Q \cap [-c,c] = \{0\}$.
	We may assume that $\partial Q = \{0\}$ by considering $Q \cap (-c,c)$ instead of $Q$.
	At least one of $P:=\{x \in Q\;|\; x>0\}$ and $Q \setminus P$ is not closed.
	We may assume that $P$ is not closed by considering $-Q$ instead of $Q$ if necessarily.
	This $P$ is the desired definable set.
	
	We show $\dim P=0$.
	For every $a \in P$, $P \cap (a,\infty)$ is discrete and closed.
	Therefore, we have $\dim (P \cap (a,\infty)) = 0$ by Lemma \ref{lem:discrete_closed}.
	We get $\dim P=0$ by Lemma \ref{lem:unbounded0}.
	
	Let $X$ be a definable discrete set.
	Put $$X \langle r \rangle :=\{x \in X\;|\; \forall x' \in X \ (x'=x) \vee (|x'-x|>r)\}$$ for $r>0$.
	Observe that $X \langle r \rangle$ is discrete and closed.
	We have $\dim X \langle r \rangle =0$ by Lemma \ref{lem:discrete_closed}.
	Observe $X \langle r \rangle \subseteq X \langle r' \rangle$ whenever $r>r'>0$.
	Since $X$ is discrete, we have $$X=\bigcup_{r>0} X\langle r \rangle.$$
	We have $$X=\bigcup_{r \in P}X\langle r \rangle$$ because $P \subseteq (0,\infty)$ and $0 \in \partial P$
	
	Consider the definable set $$T:=\{(r,y) \in P \times M\;|\; y \in X \langle r \rangle\}.$$
	For every $r \in P$, we have $\dim T_r^{\Pi_1^2} = \dim X\langle r \rangle = 0$.
	Therefore, by Definition \ref{def:dimension}(4), we get $\dim T=0$.
	Let $\pi:M^2 \to M$ be the coordinate projection onto the second coordinate.
	We have $\pi(T)=\bigcup_{r \in P}X\langle r \rangle=X$.
	By Proposition \ref{prop:dim}(2), we get $0 \leq \dim X= \dim \pi(T) \leq \dim T=0$.
\end{proof}

\begin{theorem}\label{thm:unique}
	The inequality $\myNum(\mathcal M) \leq 1$ holds for the following structures $\mathcal M$:
	\begin{enumerate}
		\item[(a)] $\mathcal M$ is a d-minimal expansion of an ordered group.
		\item[(b)] $\mathcal M$ is a weakly o-minimal expansion of an ordered field.
	\end{enumerate}
\end{theorem}
\begin{proof}
	Let $M$ be the universe of $\mathcal M$.
	
	We first treat the case in which $\mathcal M$ is a d-minimal expansion of an ordered group.
	Let $\dim$ be a dimension function.
	We show that $\dim$ coincides with the topological dimension.
	For that purpose, by Proposition \ref{prop:dim}(1), we have only to show that, for every nonempty definable subset $X$ of $M$, $\dim X=1$ if and only if $X$ has a nonempty interior.
	The `if' part follows from Lemma \ref{lem:interval1}.
	
	We prove the `only if' part.
	Let $X$ be a definable subset with an empty interior.
	By the definition of d-minimality, $X$ is a union of finitely many definable discrete sets.
	We consider two separate cases.
	Suppose that every definable discrete subset of $M$ is closed.
	We get $\dim X=0$ by Lemma \ref{lem:discrete_closed}.
	Next suppose that there exists a definable non-closed discrete subset of $M$.
	There exist definable discrete subsets $C_1,\ldots C_k$ of $M$ such that $X=\bigcup_{i=1}^k C_i$.
	We have $\dim C_i = 0$ for each $1 \leq i \leq k$ by Lemma \ref{lem:CBcase}.
	We get $\dim X=0$ by Definition \ref{def:dimension}(2).
	
	Next we treat the case where $\mathcal M$ is a weakly o-minimal expansion of an ordered field.
	Let $\dim$ be a dimension function.
	Let $C$ be a convex open definable subset of $M$.
	$C$ contains a nonempty open interval $I$.
	It is easy to construct a definable bijection between $I$ and $M$.
	We have $1 \geq \dim(C) \geq \dim(I)=\dim(M)=1$ by Definition \ref{def:dimension} and Proposition \ref{prop:dim}(2).
	Let $X$ be a definable subset of $M$.
	Since $X$ is a union of a finite set and finitely many open convex set, we have $\dim X=1$ if and only if $X$ has a nonempty interior.
	This means that $\dim$ coincides with the topological dimension on $\myDef_1(\mathcal M)$.
	These two coincide with each other on $\myDef(\mathcal M)$ by Proposition \ref{prop:dim}(1).
\end{proof}

\subsection{O-minimal case}
In this subsection, we suppose $\mathcal M$ is o-minimal. 
Our main target of this subsection is Theorem \ref{thm:omin_finite}.
For its proof, we introduce several notions which is useful for the study of dimension functions in o-minimal structures.
First, we construct a new function from $\myDef(\mathcal M)$ into $\mathbb N \cup \{-\infty\}$ from a definable subset of $M$ and another function having the same domain and codomain.

\begin{definition}\label{def:new_dim}
	Let $\mathcal M$ be a structure and $\dim:\myDef(\mathcal M) \to \mathbb N \cup \{-\infty\}$ be a function satisfying conditions (1), $(2)_1$ in Definition \ref{def:dimension} and $(4')_n$ in Definition \ref{def:dimension2} for $n>1$ in Definition \ref{def:dimension}.
	Let $I$ be a definable subset of $M$ with $\dim I=1$.
	
	Put $I_0:=I$ and $I_1 := M \setminus I$. 
	For every positive integer $n$, $0<k \leq n$ and $\tau \in \{0,1\}^n$, the $k$-th coordinate of $\tau$ is denoted by $\tau(k)$ and put $I \langle \tau \rangle:=\prod_{i=1}^{n} I_{\tau(i)}$.
	Recall that $\Pi:M^n \to M^{n-1}$ be the projection forgetting the last coordinate.
	We define the dimension function $$\mydim [ \dim, I ]:\myDef(\mathcal M) \to \mathbb N \cup \{-\infty\}$$ as follows:
	
	Let $X$ be a definable subset of $M^n$.
	\begin{itemize}
		\item We put $\mydim [ \dim, I ](X)=-\infty$ if $X=\emptyset$;
		\item Let $\tau \in \{0,1\}^n$ and assume that $\emptyset \neq X \subseteq I\langle \tau \rangle$.
		\begin{itemize}
			\item Suppose $n=1$. If $\tau(1)=0$, $\mydim [ \dim, I ](X):=\dim X$.
			
			Put $\mydim [ \dim, I ](X):=0$ if $\tau(1)=1$. 
			\item Suppose $n>1$. If $\tau(n)=0$, put 
			\begin{align*}
				&S_1(X):=\{x \in \Pi(X)\;|\; \dim X_x^{\Pi} =1 \} \text{ and }\\
				&S_0(X):=\Pi(X) \setminus S_1(X).
			\end{align*}
			We define $\mydim [ \dim, I ](X):=\max\{\dim S_0(X), \dim S_1(X)+1\}$.
			
			If $\tau(n)=1$, put $\mydim [ \dim, I ](X)=\mydim [ \dim, I ](\Pi(X))$.
		\end{itemize}
		\item If $X$ is not empty, we put $$\mydim [ \dim, I ](X):=\max\{\mydim[\dim, I](X \cap I\langle \tau \rangle)\;|\; \tau \in \{0,1\}^n\}.$$
	\end{itemize}
	
	In this paper, we mainly consider the case in which $\dim=\topdim$, where $\topdim$ is the topological dimension defined in Example \ref{ex:top_dim}.
	If $\dim=\topdim$, we write $\mydim[I]$ instead of $\mydim [ \dim, I ]$.
\end{definition}

\begin{lemma}\label{lem:new_dim}
	Let $\mathcal M$, $\dim:\myDef(\mathcal M) \to \mathbb N \cup \{-\infty\}$ and $I$ be as in Definition \ref{def:new_dim}.
	The function $\mydim[\dim, I]$ defined in Definition \ref{def:new_dim} satisfies condition $(1)$, $(2)_1$ and $(4')_n$ for $n>1$ in Definition \ref{def:dimension} and Definition \ref{def:dimension2}.
\end{lemma}
\begin{proof}
	Satisfaction of (1) and $(2)_1$ is obvious.
	We omit the proof.
	We prove $(4')_n$ by induction on $n$.
	Note that $(2)_i$ holds for $1 \leq i <n$ by Lemma \ref{lem:2n} and the induction hypothesis. 
	We abbreviate $\mydim[\dim, I]$ by $\mydim$ for the simplicity of notations.
	
	Let $X$ be a definable subset of $M^n$.
	Put $$X(i):=\{x \in \Pi(X)\;|\; \mydim X_x^{\Pi}=i\}$$ for $i=0,1$ as in Definition \ref{def:dimension2}$(4')_n$. 
	By the definition of $\mydim$, we obviously have
	\begin{align*}
		&X(1) = \{x \in \Pi(X \cap (M^{n-1} \times I)) \;|\;\dim(X \cap (M^{n-1} \times I))_x^{\Pi}=1\} \text{ and }\\
		&X(0)=\Pi(X) \setminus X(1).
	\end{align*}
	They are definable.
	
	We put $X_{\tau}:=X \cap \Pi^{-1}(I\langle \tau \rangle)$ for $\tau \in \{0,1\}^{n-1}$.
	Define $X_{\tau}(i)$ in the same manner as $X(i)$ for $i=0,1$.
	It is obvious that $X(i)=\bigcup_{\tau \in \{0,1\}^{n-1}} X_{\tau}(i)$.
	For every $\tau \in \{0,1\}^{n-1}$, let $\tau[i] \in \{0,1\}^n$ such that $\tau[i](j)=\tau(j)$ for $1 \leq j <n$ and $\tau[i](n)=i$ for $i=0,1$.
	
	Suppose $X_{\tau}$ enjoys property $(4')_n$ for every $\tau \in \{0,1\}^{n-1}$.
	By the definition of $\mydim$ and $(2)_{n-1}$, we have
	\begin{align*}
		\mydim  X &= \max \{\mydim (X \cap I\langle \tau'\rangle )\;|\; \tau' \in \{0,1\}^n\}\\
		&=\max\{ \max\{\mydim(X \cap I\langle \tau[i] \rangle )\;|\; i=0,1\}\;|\; \tau \in \{0,1\}^{n-1}\}\}\\
		&=\max\{\mydim X_{\tau}\;|\;\tau \in \{0,1\}^{n-1} \}\\
		&= \max \{\max\{\mydim X_{\tau}(i)+i\;|\; i=0,1\}\;|\; \tau  \in \{0,1\}^{n-1}\}\\
		&=\max\{\max\{\mydim X_{\tau}(i)\;|\; \tau \in \{0,1\}^{n-1}\}+i\;|\; i=0,1\}\\
		&=\max\{\mydim X(i) + i \;|\; i=0,1\}.
	\end{align*}
	Therefore, we may assume that there exists $\tau \in \{0,1\}^{n-1}$ such that $\Pi(X) \subseteq I\langle \tau \rangle $ without loss of generality.
	
	Suppose $\Pi(X) \subseteq I\langle \tau \rangle $ for some $\tau \in \{0,1\}^{n-1}$.
	Let $X_i:=X \cap I \langle \tau[i] \rangle$ for $i=0,1$.
	Define $X_0(i)$ in the same manner as $X(i)$ for $i=0,1$.
	Put $Y:=\Pi(X_1)$.
	By the definition of $\mydim$, we have the following:
	\begin{align*}
		X(1) &= X_0(1)\\
		X(0) &= X_0(0) \cup (Y \setminus X(1))
	\end{align*}
	It is obvious from the definition of $\mydim$ that 
	\begin{align*}
		\mydim X &= \max\{\mydim X_i\;|\; i=0,1\}\\
		&=\max\{\mydim X_0(1)+1,\mydim X_0(0), \mydim Y\}\\
		&=\max\{\mydim X(1)+1,\mydim X_0(0), \mydim Y\}.
	\end{align*}
	
	We consider three separate cases.
	First suppose $\mydim Y > \mydim X(1)$.
	We have $\mydim (Y \setminus X(1))=\mydim Y$ by $(2)_{n-1}$.
	We have $\mydim X(0)=\max\{\mydim X_0(0),\mydim Y\}$ by $(2)_{n-1}$.
	This implies $\mydim X = \max\{\mydim X(0),\mydim X(1)+1\}$.
	
	Next suppose $\mydim Y \leq \mydim X(1)<\mydim X_0(0)$.
	We have $\mydim (Y \setminus X(1)) < \mydim Y < \dim X_0(0)$.
	This implies $\mydim X(0) = \mydim X_0(0)$.
	Therefore, 
	\begin{align*}
	&\max\{\mydim X(0),\mydim X(1)+1\} = \max\{\mydim X_0(0) ,\mydim X(1)+1\}\\
	&=\max\{\mydim X_0(0), \mydim Y ,\mydim X(1)+1\}=\mydim X.
	\end{align*}
	
	Suppose $\mydim Y \leq \mydim X(1)$ and $\mydim X_0(0) \leq \mydim X(1)$.
	By $(2)_{n-1}$, we get $\mydim X(0) \leq \max\{\dim X_0(0), \dim Y\} \leq \dim X(1)$.
	We get 
	\begin{align*}
		\mydim(X)&=\max\{\mydim X(1)+1,\mydim X_0(0), \mydim (Y)\}\\
		&=\mydim(X(1))+1=\max\{\mydim X(0), \mydim X(1)+1\}.
	\end{align*}
\end{proof}

\begin{lemma}\label{lem:omin_basic}
	The dimension function on an o-minimal structure is completely determined by its restrictions to the set of all bounded open intervals.
\end{lemma}
\begin{proof}
	This lemma follows from Corollary \ref{cor:unbounded0} and o-minimality.
\end{proof}

\begin{lemma}\label{lem:omin_dimform}
	Let $\mathcal M=(M,<,\ldots)$ be an o-minimal structure.
	Every dimension function on $\mathcal M$ is of the form $\mydim[I]$ for some definable subset $I$ of $M$.
\end{lemma}
\begin{proof}
	Let $\dim$ be a dimension function on $\mathcal M$.
	Put $$I:=\{a \in M\;|\; \forall b \ (a<b) \rightarrow \dim((a,b))=1\}.$$
	$I$ is definable by Definition \ref{def:dimension}(4).
	\medskip
	
	\textbf{Claim 1.} For every $a \in I$, there exists $b>a$ such that $(a,b) \subseteq I$.
	\begin{proof}[Proof of Claim 1]
		Let $a \in I$.
		Assume for contradiction that such $b$ does not exist.
		Because of o-minimality, there exists $c>a$ such that $(a,c) \cap I=\emptyset$. 
		Since $a \in I$, we have $\dim ((a,c))=1$.
		On the other hand, we have $\dim ((d,c))=0$ for every $a<d<c$.
		By Lemma \ref{lem:unbounded0}, we have $\dim((a,c))=0$, which is absurd.
	\end{proof}
	
	We want to show that $\dim=\mydim[I]$.
	For that purpose, by Lemma \ref{lem:omin_basic}, we have only to show that, for every bounded open interval $J$, $\dim(J)=1$ if and only if $I \cap J$ has a nonempty interior.
	The `if' part is easy from the definition of $I$.
	We show the `only if' part.
	Suppose $I \cap J$ has an empty interior.
	We want to show that $\dim J=0$.
	
	We show that $I \cap J$ is an empty set.
	Assume for contradiction that $I \cap J$ is not empty and has an empty interior.
	Because of o-minimality, $I \cap J$ is a finite set.
	Choose $a \in I \cap J$.
	Then we can choose $b>a$ such that $b \in J$ and $(a,b) \cap I=\emptyset$, which contradicts Claim 1.
	
	Let $J=(a,b)$, where $a,b \in M$.
	By Claim 1 and $I \cap J=\emptyset$, we have $a \notin I$.
	There exists $c>a$ such that $\dim((a,c))=0$ by the definition of $I$.
	We may assume $c<b$ by choosing smaller $c$ by Definition \ref{def:dimension}(2).
	Consider the definable set $$X:=\{d \in M\;|\; a<d<b, \dim((a,d))=0\}.$$
	The set $X$ is not empty because $c \in X$.
	Let $e:=\sup X \in M \cup \{+\infty\}$.
	By Lemma \ref{lem:unbounded0}, we have $\dim((a,e))=0$.
	We show $e=b$.
	Assume for contradiction $e<b$.
	Observe that $e \in J$, and $e \notin I$ because $I \cap J=\emptyset$.
	Since $e \notin I$, we can take $e<f<b$ such that $\dim((e,f))=0$.
	We have $\dim((a,f))=\max\{\dim((a,e)),\dim(\{e\}),\dim((e,f))\}=0$ by Definition \ref{def:dimension}(2), which contradicts the definition of $e$.
	We have shown $\sup X=b$.
	By Lemma \ref{lem:unbounded0}, we get $\dim J=\dim((a,b))=0$.
%
\end{proof}

Let $\mathcal M=(M,<,\ldots)$ be an o-minimal structure.
A definable subset $I$ of $M$ is called \textit{self-sufficient} if, it is infinite and, for every infinite definable sets $J_1 \subseteq I$ and $J_2 \subseteq M \setminus I$, there does not exist a definable bijection between $J_1$ and $J_2$. 
Observe that, if $M \setminus I$ has an empty interior, $I$ is self-sufficient.
\begin{lemma}\label{lem:self}
	Let $\mathcal M=(M,<,\ldots)$ be an o-minimal structure.
	An infinite definable subset $I$ of $M$ is self-sufficient if and only if, for every definable subset $J$ of $I$ and every definable function $f:J \to M \setminus I$, the image $f(J)$ is a finite set.
\end{lemma}
\begin{proof}
	Suppose $I$ is self-sufficient.
	Let $J$ be an infinite definable subset of $I$ and $f:J \to M \setminus I$ be a definable function.
	By the monotonicity theorem \cite[Chapter 3, Theorem 1.2]{vdD}, there exist a finite subset $F$ of $J$ and finitely many open intervals $J_1,\ldots, J_l$ such that $J_i \subseteq J$ and the restriction $f|_{J_i}$ of $f$ to $J_i$ is monotone and continuous for $1 \leq i \leq l$.
	For each $1 \leq i \leq l$, if $f|_{J_i}$ is strictly monotone, $f|_{J_i}$ is a definable bijection from $J_i$ onto the infinite definable set $f(J_i)$, which contradicts that $I$ is self-sufficient.
	This implies $f|_{J_i}$ is constant, and $f(J_i)$ is a singleton.
	The image $f(J)=f(F) \cup \bigcup_{i=1}^l f_i(J_i)$ is a finite set.
	
	We can easily prove the contraposition of the `if' part.
	We omit the proof.
\end{proof}

\begin{lemma}\label{lem:omin_equiv}
	Let $\mathcal M=(M,<,\ldots)$ be an o-minimal structure.
	Let $I \subseteq M$ be a nonempty definable set having nonempty interior.
	The function $\mydim[I]$ is a dimension function if and only if $I$ is self-sufficient.
\end{lemma}
\begin{proof}
	For simplicity, we put $\dim = \mydim[I]$.
	
	First suppose that $I$ is not self-sufficient.
	There exist infinite definable sets $J_1 \subseteq I$, $J_2 \subseteq M \setminus I$ and a definable bijection $f:J_1 \to J_2$.
	By o-minimality and the definition of the topological dimension, we have $\dim J_1=1$.
	The $f$ witnesses that $\dim$ violates Proposition \ref{prop:dim}(2) because $\dim J_1=1$ and $\dim J_2=0$.
	This means that $\dim$ is not a dimension function.
	
	Next suppose that $I$ is self-sufficient.
	Let $X$ be an arbitrary definable subset of $M^2$.
	We have only to show that $\dim X = \dim X^{\text{switch}}$ by Proposition \ref{prop:weak} and Lemma \ref{lem:new_dim}.
	
	Observe that $\dim$ enjoys the properties (2) and $(4)$ in Definition \ref{def:dimension} by Lemma \ref{lem:2n} and Lemma \ref{lem:new_dim}.
	By the definable cell decomposition theorem \cite[Chapter 3, Theorem 2.11]{vdD} and Definition \ref{def:dimension}(2), we may assume that $X$ is a cell contained in one of $I \times I$, $I \times (M \setminus I)$, $(M \setminus I) \times I$ and $(M \setminus I) \times (M \setminus I)$.
	If $X$ is contained in $I \times I$, then the equality $\dim X = \dim X^{\text{switch}}$ follows from the fact that $\topdim$ is a dimension function.
	By Definition \ref{def:dimension}(4), we also have $\dim X = \dim X^{\text{switch}}=0$ if $X$ is contained in $(M \setminus I) \times (M \setminus I)$.
	
	Next we consider the case in which $X$ is contained in $I \times (M \setminus I)$.
	If $\Pi_1^2(X)$ is a singleton, we easily deduce that $\dim X = \dim X^{\text{switch}}=0$.
	Let us consider the case in which $\Pi_1^2(X)$ is an open interval, say $J$.
	We have two separate cases.
	First we consider the case where $X$ is the graph of definable continuous function $f:J \to M$.
	By Lemma \ref{lem:self}, $f(J)$ is a finite set.
	$f(J)$ is a singleton because $f$ is continuous and $J$ is definably connected.
	This means that $f$ is a constant function.
	It is easy to show that $\dim X = \dim X^{\text{switch}}=1$.
	We omit the verification of this equality.
	
	Let us consider the case in which there are exist definable continuous functions $f_1,f_2: J \to M \cup \{\pm \infty\}$ such that $f_1<f_2$ and $X:=\{(x,y) \in J \times M\;|\; f_1(x)<y<f_2(x)\}$.
	Here, we say that $f:J \to M \cup \{\pm \infty\}$ is continuous if either $f$ is constantly $+\infty$, constantly $-\infty$ or is a continuous function whose target is $M$.
	In the same manner as the previous paragraph, $f_1$ and $f_2$ should be constant functions.
	It is easy to show that $\dim X = \dim X^{\text{switch}}=1$.
	We omit the proof.
	
	Finally, let us consider the case in which $X$ is contained in $(M \setminus I) \times I$.
	We have $X^{\text{switch}} \subseteq I \times (M \setminus I)$.
	We have already shown that $\dim X^{\text{switch}}=\dim (X^{\text{switch}})^{\text{switch}}$ in the previous case.
	We get $\dim X = \dim X^{\text{switch}}$ because $(X^{\text{switch}})^{\text{switch}}=X$.
\end{proof}

Let $\mathcal M=(M,<,\ldots)$ be an o-minimal structure.
We define the equivalence relation $\myfineq$ on $\myDef_1(\mathcal M)$ by $I_1 \myfineq I_2$ if and only if $(I_1 \setminus I_2) \cup (I_2 \setminus I_1)$ is a finite set for  $I_1,I_2 \in \myDef_1(\mathcal M)$.
The following lemma is easy to prove.
\begin{lemma}\label{lem:fineq}
Let $I_1,I_2 \in \myDef_1(\mathcal M)$ with $I_1 \myfineq I_2$.
$I_1$ is self-sufficient if and only if $I_2$ is so.
\end{lemma}
Let $\mySelf(\mathcal M)$ be the set of the equivalence class of self-sufficient sets under the equivalence relation $\myfineq$.

\begin{lemma}\label{lem:well_defined}
	Let $\mathcal M$ be an o-minimal structure.
	For self-sufficient sets $I_1$ and $I_2$, the equality $\mydim[I_1]=\mydim[I_2]$ holds if and only if $I_1 \myfineq I_2$.
\end{lemma}
\begin{proof}
	It is easy to show that, for self-sufficient sets $I_1$ and $I_2$, the equality $$\mydim[I_1]|_{\myDef_1(\mathcal M)}=\mydim[I_2]|_{\myDef_1(\mathcal M)}$$ holds if and only if $I_1 \myfineq I_2$, where $\mydim[I_i]|_{\myDef_1(\mathcal M)}$ is the restriction of $\mydim[I_i]$ to $\myDef_1(\mathcal M)$ for $i=1,2$.
	Since $\mydim[I_1]$ and $\mydim[I_2]$ are dimension functions by Lemma \ref{lem:omin_equiv}, we have the equality $\mydim[I_1]=\mydim[I_2]$ holds if and only if $I_1 \myfineq I_2$ by Proposition \ref{prop:dim}(1).
\end{proof}

\begin{proposition}\label{prop:omin_equiv}
	Let $\mathcal M$ be an o-minimal structure.
	Then, there exists a one-to-one correspondence of $\mySet(\mathcal M)$ with $\mySelf(\mathcal M)$. 
\end{proposition}
\begin{proof}
	Let $\iota: \mySelf(\mathcal M) \to \mySet(\mathcal M)$ given by $\iota([I])=\mydim[I]$, where $[I]$ is the equivalence class of the self-sufficient set $I$.
	The map $\iota$ is well-defined by Lemma \ref{lem:omin_equiv} and Lemma \ref{lem:well_defined}.
	The injectivity of $\iota$ easily follows from Lemma \ref{lem:well_defined}.
	The surjectivity of $\iota$ follows from Lemma \ref{lem:omin_dimform} and Lemma \ref{lem:omin_equiv}.
\end{proof}

Thanks to Proposition \ref{prop:omin_equiv}, if $\mathcal M$ is o-minimal structure, we may investigate the set of self-sufficient sets $\mySelf(\mathcal M)$ instead of the set of dimension functions $\mySet(\mathcal M)$ in order to compute $\myNum(\mathcal M)$.
The following is the basic properties of self-sufficient sets.

\begin{lemma}\label{lem:self_sufficient}
	Let $\mathcal M=(M,<,\ldots)$ be an o-minimal structure.
	Let $I, I_1, I_2 \in \myDef_1(\mathcal M)$.
	\begin{enumerate}
		\item[(1)] $M \setminus I$ is self-sufficient if $I$ is so and $M \setminus I$ is infinite.
		\item[(2)] $I_1 \cap I_2$ is self-sufficient if $I_i$ is so for $i=1,2$ and $I_1 \cap I_2$ is infinite.
		\item[(3)] $I_1 \cup I_2$ is self-sufficient if $I_i$ is so for $i=1,2$.
	\end{enumerate}
\end{lemma}
\begin{proof}
	(1) The proof is easy. We omit it.
	
	(2) Let $J_1 \subseteq I_1 \cap I_2$ be an infinite definable set.
	Let $J_2 \subset M \setminus (I_1 \cap I_2)$ be an infinite definable set.
	By the definition of self-sufficient sets, we only have to show that there does not exist a definable bijection between $J_1$ and $J_2$.
	Assume for contradiction that there exists a definable bijection $f:J_1 \to J_2$.
	Put $K_1:= J_2 \cap I_2$ for $i=1,2$ and $K_2:=J_2 \setminus K_1$.
	Observe that $K_i \subseteq M \setminus I_i$ for $i=1,2$.
	At least one of $K_1$ and $K_2$ is infinite.
	Let $k \in \{1,2\}$ such that $K_k$ is infinite.
	The restriction $f|_{f^{-1}(K_k)}:f^{-1}(K_k) \to K_k$ is a definable bijection between infinite definable sets $f^{-1}(K_k)$ and $K_k$, which contradicts that $I_k$ is self-sufficient.
	
	(3) If $M \setminus (I_1 \cup I_2)$ has an empty interior, it is self-sufficient.
	The remaining case follows from (1) and (2) of this lemma.
\end{proof}

Let $\mathcal M=(M,<,\ldots)$ be a o-minimal structure.
A self-sufficient definable set $I$ is called \textit{minimally self-sufficient} if, for every self-sufficient set $X$, either $I \cap X$ is a (possibly empty) finite set or there exists a finite subset $F$ of $M$ such that $I \subseteq X \cup F$. 
A finite family $\mathcal C=\{C_1,\ldots, C_m\}$ of minimally self-sufficient sets called a \textit{generating system} if the following conditions are satisfied:
\begin{itemize}
	\item $C_i \cap C_j$ is a (possibly empty) finite set for $i \neq j$;
	\item $M \setminus \bigcup_{i=1}^m C_i$ is a finite set.
\end{itemize} 

\begin{lemma}\label{lem:omin_finite1}
	Let $\mathcal M=(M,<,\ldots)$ be an o-minimal structure.
	Suppose $|\mySelf(\mathcal M)|<\infty$.
	Then, $\mathcal M$ has a generating system $\{C_1, \ldots, C_m\}$ of minimally self-sufficient sets.  
\end{lemma}
\begin{proof}
	We denote $\mySelf(\mathcal M)$ by $\mySelf$ for short.
	We write the equivalence class of $X \in \myDef_1(\mathcal M)$ under the equivalence relation $\myfineq$ by $[X]$. 
	For any $X_1,X_2 \in \myDef_1(\mathcal M)$, we write $X_1 \preceq X_2$ if there exists a finite subset $F$ of $M$ such that $X_1 \subseteq X_2 \cup F$.
	\medskip
	
	\textbf{Claim 1.} For any $[X] \in \mySelf$, there exists a minimally self-sufficient set $Z$ such that $Z \preceq X$.
	\begin{proof}[Proof of Claim 1]
		Assume for contradiction that such a minimally self-sufficient set does not exist.
		Let $X_1 = X$.
		Since $X_1$ is not minimally self-sufficient, there exists $[Y_1] \in \mathfrak D$ such that $X_2:=X_1 \cap Y_1$ is infinite and $X_1 \not\preceq Y_1$.
		$X_2$ has a nonempty interior because of o-minimality.
		We get $[X_2] \in \mySelf$ by Lemma \ref{lem:self_sufficient}(2).
		Observe that $X_1 \not\myfineq X_2$ and $X_2 \subseteq X_1$.
		In the same manner, we can construct infinitely many definable subsets $X_1,X_2,\ldots$ of $M$ such that $[X_i] \in \mySelf$, $X_i \not\myfineq X_{i+1}$ and $X_i \subseteq X_{i+1}$ for $i \geq 1$.
		We can easily show that $X_i \not\myfineq X_j$ for $i \neq j$.
		This contradicts the assumption that $\mySelf$ is a finite set.
	\end{proof}

	Let $\{[C_i]\;|\;1 \leq i \leq m\}$ be the family of the equivalence class of minimally self-sufficient sets, which exists thanks to Claim 1.
	First, $C_i \cap C_j$ is a finite set for $i \neq j$.
	In fact, if $C_i \cap C_j$ is an infinite set, $[C_i \cap C_j] \in \mySelf$ by Lemma \ref{lem:self_sufficient}(2).
	By minimality, we get $[C_i]=[C_i \cap C_j]$ and $[C_i \cap C_j]=[C_j]$.
	This implies $[C_i]=[C_j]$, which is absurd.
	Next, we show $M \setminus (\bigcup_{i=1}^m C_i)$ is a finite set.
	Assume for contradiction that $M \setminus (\bigcup_{i=1}^m C_i)$ is an infinite set.
	We have $[\bigcup_{i=1}^m C_i] \in \mySelf$ by Lemma \ref{lem:self_sufficient}(3).
	Since $M \setminus (\bigcup_{i=1}^m C_i)$ has a nonempty interior, Lemma \ref{lem:self_sufficient}(1) implies $[M \setminus (\bigcup_{i=1}^m C_i)] \in \mySelf$.
	By Claim 1, there exists a minimal element $Z$ such that $Z \preceq M \setminus (\bigcup_{i=1}^m C_i)$, which is absurd.
	We have shown that  $\mathcal M$ has a generating system of minimally self-sufficient sets.  
\end{proof}

\begin{lemma}\label{lem:omin_finite2}
	Let $\mathcal M=(M,<,\ldots)$ be an o-minimal structure having a generating system $\{C_1, \ldots, C_m\}$ of minimally self-sufficient sets.  
	Then, $|\mySelf(\mathcal M)|=2^m-1$.
\end{lemma}
\begin{proof}
	We denote $\mySelf(\mathcal M)$ by $\mySelf$ for short.
	Consider the map $$\iota:\{\emptyset \neq \mathcal I \subseteq \{1,2,\ldots,m\}\} \to \mySelf$$ given by $\iota(\mathcal I)=[\bigcup_{i \in \mathcal I}C_i]$.
	Recall that $[X]$ is the equivalence class of $X \in \myDef_1(\mathcal M)$ under $\myfineq$.
	Observe that $[\bigcup_{i \in \mathcal I}C_i]$ is an element of $\mySelf$ by Lemma \ref{lem:self_sufficient}(3).
	$\iota$ is injective because $C_i \cap C_j$ is a finite set for $i \neq j$.
	We show that $\iota$ is surjective.
	Let $[X] \in \mySelf$.
	Let $\mathcal I:=\{1 \leq i \leq m\;|\; X \cap C_i \text{ is an infinite set}\}$.
	By minimality of $C_i$, we have $X \cap C_i \myfineq C_i$ for $i \in \mathcal I$.
	The relation $X \myfineq \bigcup_{i \in \mathcal I}C_i$ follows from the relation $M \myfineq \bigcup_{i=1}^m C_i$ and minimality of $C_i$.
	We have shown that $[X]=\iota(\mathcal I)$, which means the surjectivity of the map $\iota$.
	
	We have $|\mySelf|=|\{\emptyset \neq \mathcal I \subseteq \{1,2,\ldots,m\}\}|=2^m-1$.
\end{proof}

The following theorem is easily deduced from what we have proved.
\begin{theorem}\label{thm:omin_finite}
	Let $\mathcal M$ be an o-minimal structure such that $\myNum(\mathcal M)$ is finite.
	Then there exist a positive integer $m$ such that $\myNum(\mathcal M)=2^m-1$.
\end{theorem}
\begin{proof}
	We obtain this theorem from Proposition \ref{prop:omin_equiv}, Lemma \ref{lem:omin_finite1} and Lemma \ref{lem:omin_finite2}.
\end{proof}

Models of DLO are examples of o-minimal structures.
By Example \ref{ex:top_dim}, the topological dimension on $\mathcal M$ is a dimension function.
The following proposition says a model of DLO has infinitely many dimension functions.

\begin{proposition}\label{prop:DLO}
	The equality $$\myNum(\mathcal M)=|M|$$ holds for every model $\mathcal M=(M,<)$ of DLO.
\end{proposition}
\begin{proof}
	DLO admits quantifier elimination (cf.\ \cite[Theorem 3.3.2]{TZ}), and DLO is an o-minimal theory.
	
	Put $I_a:=\{x \in M\;|\;x>a\}$.
	We want to show that the function $\mydim[I_a]$ defined in Definition \ref{def:new_dim} is a dimension function.
	We have only to prove that $I_a$ is self-sufficient by Lemma \ref{lem:omin_equiv}.
	Assume for contradiction there exist nonempty open intervals $J_1 \subseteq I_a$, $J_2 \subseteq M \setminus I_a$ and a definable bijection $f:J_1 \to J_2$.
	By the monotonicity theorem \cite[Chapter 3, Theorem 1.2]{vdD}, we may assume that $f$ is continuous and monotone by shrinking $J_1$ and $J_2$ if necessary.
	At every point in $I_a \times (M \setminus I_a)$, the formulas $x_1<x_2$ and $x_1=x_2$ fail, and the formula $x_1>x_2$ holds, where $x_1$ and $x_2$ are free variables. 
	Since DLO admits quantifier elimination and the graph of $f$ is contained in $I_a \times (M \setminus I_a)$, there exist finitely many formulas $\phi_{ij}$ of the forms $$x_i \square a,$$ where $i=1,2$, $\square \in \{<,>,=\}$ and $a \in M$, such that the graph of $f$ is given by the formula $\bigvee_{i=1}^n \bigwedge_{j=1}^{m_i}\phi_{ij}$.
	It is easy to show that this formula cannot define the graph of a strictly monotone function, which is absurd.
	We omit the details.
	\medskip
	
	We prove that $\myNum(\mathcal M)=|M|$.
	Since $\mydim[I_a] \in \mySet(\mathcal M)$ for $a \in M$, we have $|M| \leq \myNum(\mathcal M)$.
	We construct an injective map $\iota:\mySet(\mathcal M) \to \myDef_2(\mathcal M)$.
	Let $\dim$ be a dimension function on $\mathcal M$.
	We define $\iota(\dim) \in \myDef_2(\mathcal M)$ as follows:
	\begin{align*}
		&\iota(\dim):=\{(a,b) \in M^2\;|\; \dim ((a,b))=0\}.
	\end{align*}
	Observe that $\iota(\dim)$ is a definable set by Definition \ref{def:dimension}(4).
	Lemma \ref{lem:omin_basic} implies that $\iota$ is injective.
	$\myDef_2(\mathcal M)$ is of the cardinality smaller than or equal to the cardinality of the set of finite sequences of elements from the union of $M$ with the finite set consisting of logical symbols, two variables and $<$. 
	Therefore, we have $|\myDef_2(\mathcal M)|=|\mathcal M|$ (under the assumption of the axiom of choice).
	We have $\myNum(\mathcal M)=|M|$.
\end{proof}

We give an o-minimal example such that $\myNum(\mathcal M)=2^{m+1}-1$ for every positive integer $m$.
This is a concatenation of o-minimal structures $\mathcal M_1,\ldots, \mathcal M_m$ with $\myNum(\mathcal M_i)=1$ for $1 \leq i \leq m$.
\begin{proposition}\label{prop:omin_finite_example}
	For $m>0$, let $\mathcal L_m$ be the expansion of the language of ordered groups $\{<,+\}$ by $m-1$ constant symbols $c_1,\ldots,c_{m-1}$.
	Put $c_0=-\infty$ and $c_m=\infty$.
	Let $T_m$ be the $\mathcal L_m$-theory describing the following:
	\begin{itemize}
		\item The axiom of DLO;
		\item $c_{i-1}<c_i$ for $1<i<m$
		\item For every $1 \leq i \leq m$, $((c_{i-1},c_i),+|_{\prod_{j=1}^2(c_i,c_{i+1})})$ is a divisible Abelian group; 
		\item The restriction of the function $+$ off $\bigcup_{i=1}^m \prod_{j=1}^2 (c_{i-1},c_i)$ is constantly equal to $c_1$.
	\end{itemize}
	Then, $T_m$ is a consistent o-minimal theory, that is, every model of $T_m$ is  o-minimal, and the equality $$\myNum(\mathcal M)=2^{m+1}-1$$ holds for every model $\mathcal M$ of $T_m$.
\end{proposition}
\begin{proof}
	First we show $T_m$ is a consistent theory.
	We construct a model $\mathcal M=(M,<^{\mathcal M},+^{\mathcal M},c_1^{\mathcal M},\ldots, c_{m-1}^{\mathcal M})$ of $T_m$.
	Let $Q_i:=\mathbb Q$ for $1 \leq i \leq m$ and let $S_i$ be singletons for $1 \leq i < m$.
	Put $$M:=Q_1 \sqcup S_1 \sqcup Q_2 \sqcup S_2 \sqcup \cdots \sqcup S_{m-1} \sqcup Q_m,$$
	where $\sqcup$ represents disjoint unions.
	We interpret the symbols $c_i$, $<$ and $+$ as follows:
	\begin{itemize}
		\item $c_i^{\mathcal M} \in S_i$ for $1 \leq i <m$.
		\item $c_i^{\mathcal M}<^{\mathcal M}c_j^{\mathcal M}$ if and only if $1 \leq i<j < m$;
		\item For $a \in Q_i$, $a <^{\mathcal M} c_j^{\mathcal M}$ if and only if $1 \leq i \leq j<m$;
		\item For $a \in Q_i$ and $b \in Q_j$, $a<^{\mathcal M}b$ if and only if $i=j$ and $a<b$ as elements in $\mathbb Q$ or $i<j$;
		\item If $a,b \in Q_i$, $a+^{\mathcal M}b = a+_{\mathbb Q}b \in Q_i$, where $+_{\mathbb Q}$ is the addition in $\mathbb Q$;
		\item Otherwise, we put $a+^{\mathcal M}b=c_1^{\mathcal M}$.
	\end{itemize}
	It is easy to check that $\mathcal M=(M,<^{\mathcal M},+^{\mathcal M},c_1^{\mathcal M},\ldots, c_{m-1}^{\mathcal M})$ is a model of $T_m$.
	
	Observe that every divisible Abelian group is naturally a $\mathbb Q$-vector space.
	We expand $\mathcal L_m$ by the collection of unary function symbols $\{\lambda_q\}_{q \in \mathbb Q}$ and interpret $\lambda_q$ as the multiplication by a rational $q$ on $(c_{i-1},c_i)$ for $1 \leq i \leq m$.
	Let $\mathcal L_m'$ be the expansion denoted of $\mathcal L_m$.
	Let $T'_m$ be the extension of the theory $T_m$ adding the sentences $\lambda_q(c_i)=c_1$ and $\forall x  \in (c_{j-1},c_j)\ \lambda_q(x)=q \cdot x$ for every $q \in \mathbb Q$, $1 \leq i \leq m-1$ and $1 \leq j \leq m$.
	
	We want to show that $T_m'$ admits quantifier elimination.
	Let $\mathcal N_j$ be a models of $T_m'$ for $j=1,2$.
	Let $\mathcal A$ be their common substructure.
	We denote the universes of $\mathcal N_1$, $\mathcal N_2$ and $\mathcal A$ by $N_1$, $N_2$ and $A$, respectively.
	Let $\phi$ be an arbitrary primitive existential formula with parameters from $A$.
	In other words, there exist finitely many basic formulas $\phi_1(x), \ldots, \phi_l(x)$ with the single free variable $x$ such that $$\phi = \exists x\ \bigwedge_{i=1}^l \phi_i(x).$$
	By a QE test (cf.\ \cite[Theorem 3.2.5]{TZ}), we have only to show that $\mathcal N_2 \models \phi$ whenever $\mathcal N_1 \models \phi$.
	
	Let $n_1 \in N_1$ such that $\mathcal N_1 \models \bigwedge_{i=1}^l \phi_i(n_1)$.
	We want to find $n_2 \in N_2$ such that $\mathcal N_2 \models \bigwedge_{i=1}^l \phi_i(n_2)$.
	If $n_1=c_k^{\mathcal N_1}$ for some $1 \leq k <m$, $n_2:=c_k^{\mathcal N_2}$ satisfies the given formula.
	
	Let us consider the case in which $n_1 \in (c_{k-1},c_k)$ for some $1 \leq k \leq m$. 
	By considering $\bigwedge_{i=1}^l \phi_i(x) \wedge (x>c_{k-1}) \wedge (x<c_k)$ instead of $\bigwedge_{i=1}^l \phi_i(x)$, we may assume that every $x$ satisfying the formula $\bigwedge_{i=1}^l \phi_i(x)$ belongs to the interval $(c_{k-1},c_k)$.
	Since we may assume that $x \in (c_{k-1},c_k)$, every basic formula $\phi_i(x)$ for $1 \leq i \leq l$ is equivalent to one of the following forms or their negation:
	\begin{enumerate}
		\item[(*)] $x *_i a_i$, where $a_i \in A$ and $*_i \in \{=, <, >\}$;
	\end{enumerate}
	We consider the following three separate cases.
	\medskip
	
	\textbf{Case 1.} Suppose $n_1 \in A$. Put $n_2=n_1$ in this case.
	\medskip
	
	\textbf{Case 2.} Suppose $n_1 \notin A$.
	Enumerate elements in $A \cap (c_{k-1},c_k)$ appearing in the formulas $\phi_i(x)$ in the increasing order, say $a_1,\ldots,a_t$.
	If $n_1 <a_1$, choose an element $n_2 \in N_2$ with $c_{k-1}<n_2<a_1$, which is possible because $(c_{k-1},c_k)$ is a model of DLO.
	If $n_1>a_t$, choose an element $n_2 \in N_2$ with $a_t<n_2<c_k$.
	Finally, if there exists $1 < p \leq t$ such that $a_{p-1}<n_1<a_p$, put $n_2=(a_{p-1}+a_p)/2$.
	We have $\mathcal N_2 \models \bigwedge_{i=1}^l \phi_i(n_2)$ in every case.
	We have shown that $T'_m$ admits quantifier elimination.
	\medskip
	
	Using this QE result, it is easy to show that $\mathcal M$ is o-minimal for every model $\mathcal M$ of $T_m$.
	
	Fix a model $\mathcal M=(M,<^{\mathcal M}, +^{\mathcal M},c_1^{\mathcal M},\ldots,c_{m-1}^{\mathcal M})$ of $T_m$.
	We drop the superscript $\mathcal M$ of the symbols from now on.
	Put $C_i:=(c_{i-1},c_i)$ for $1 \leq i \leq m$.
	We want to show the equality $\myNum(\mathcal M)=2^{m+1}-1$.
	For that purpose, we have only to prove that $\{C_i\;|\; 1 \leq i \leq m\}$ is a generating system of minimally self-sufficient sets by Proposition \ref{prop:omin_equiv} and Lemma \ref{lem:omin_finite2}.
	The relations $C_i \cap C_j=\emptyset$ for $1 \neq j$ and $M \myfineq \bigcup_{i=1}^m C_i$ obviously hold.
	The remaining task is to show that $C_i$ is a minimally self-sufficient set for each $1 \leq i \leq m$.
	For short, we denote $C_i$ by $C$.
	
	First we show that $C$ is self-sufficient.
	The proof is similar to that in Proposition \ref{prop:DLO}.
	Assume for contradiction there exist nonempty open intervals $J_1 \subseteq C$, $J_2 \subseteq M \setminus C$ and a definable bijection $f:J_1 \to J_2$.
	We have either $a <c$ for every $c \in C$ and $a \in J_2$ or $a >c$ for every $c \in C$ and $a \in J_2$.
	In the first case, we can lead to a contradiction completely in the same manner by replacing $I_a$ in the proof of Proposition \ref{prop:DLO} by $C$.
	The proof for the latter case is similar, and we omit the proofs.
	
	Finally, we show that $C$ is a minimally self-sufficient.
	Let $X$ be a self-sufficient set such that $X \cap C$ is an infinite set.
	We can take a bounded open interval $J_1=(a_1,a_2)$ contained in $X \cap C$ because $\mathcal M$ is o-minimal.
	We have only to show that $C \setminus X$ is a finite set.
	Assume for contradiction that $C \setminus X$ is an infinite set.
	By o-minimality, there exists a bounded open interval $J_2'=(b_1,b_2)$ contained in $C \setminus X$.
	We may assume that $a_2-a_1 \leq b_2-b_1$ by shrinking $J_1$ if necessary.
	Put $J_2:=((b_1+b_2)/2-(a_2-a_1)/2, (b_1+b_2)/2+(a_2-a_1)/2)$.
	Consider the definable bijection $f:J_1 \ni x \mapsto x+(b_1+b_2)/2-(a_1+a_2)/2 \in J_2$.
	This contradicts the assumption that $X$ is self-sufficient.
	We have shown that $C$ is a minimally self-sufficient.
\end{proof}

\subsection{Weakly o-minimal case}
Theorem \ref{thm:omin_finite} implies that there exist infinitely many positive integers $m$ such that there exist no o-minimal structures $\mathcal M$ such that $\myNum(\mathcal M)=m$.
In contrast to it, for every positive integer $m$, there exists a weakly o-minimal structure $\mathcal M$ such that $\myNum(\mathcal M)=m$.
Every nonvaluational weakly o-minimal expansion of an ordered field is an example of $\myNum(\mathcal M)=1$ by \cite{A}, \cite[Corollary 4.4]{MMS}, \cite[Theorem 4.2]{Wencel}, \cite[Theorem 2.15]{Wencel2} and Theorem \ref{thm:unique}(2).
We look for a weakly o-minimal structure $\mathcal M$ such that $\myNum(\mathcal M)=m$ with $m>1$.

\begin{theorem}\label{thm:weakly}
	For $m>0$, let $\mathcal L_m$ be the expansion of the language of ordered groups $\{<,+\}$ by $m$ unary predicates $U_1, \ldots, U_m$.
	Let $T_m$ be the $\mathcal L_m$-theory describing the followings:
	\begin{itemize}
		\item The axioms of ordered divisible Abelian group;
		\item The axioms saying that $U_m$ is a nontrivial convex subgroup of the universe;
		\item The axioms saying that $U_{k-1}$ is a nontrivial convex subgroup of $U_k$ for $k>1$.
	\end{itemize}
	Here, we identified the unary predicate $U_i$ with the set of realizations of $U_i$ for $1 \leq i \leq m$.
	Then, $T_m$ is a consistent weakly o-minimal theory, that is, every model of $T_m$ is weakly o-minimal, and the equality $$\myNum(\mathcal M)=m+1$$ holds for every model $\mathcal M$ of $T_m$.
\end{theorem}
\begin{proof}
	We identify the unary predicate $U_i$ with the set of realizations of $U_i$ for $1 \leq i \leq m$ through the proof.
	
	First we prove that $T_m$ is consistent.
	$(\mathbb Q,<_{\mathbb Q},+_{\mathbb Q})$ is the ordered divisible Abelian group.
	Put $Q:=\mathbb Q^{m+1}$.
	$Q$ is naturally a divisible Abelian group with the addition operator $+^Q$.
	Let $<^Q$ be the lexicographic order on $Q$, then $(Q,<^Q,+^Q)$ is an ordered divisible Abelian group.
	For $1 \leq k \leq m$, put $U_k^Q:=\{0_{m+1-k}\} \times \mathbb Q^k$, where $0_{m+1-k}$ is the origin of $\mathbb Q^{m+1-k}$.
	Then, $(Q,<_Q,+^Q,U_1^Q,\ldots, U_m^Q)$ is a model of $T_m$.
	This means that $T_m$ is consistent.

	Every model $\mathcal M$ of $T_m$ is naturally a $\mathbb Q$-vector space, and the set of realization $U_k^{\mathcal M}$ of $U_k$ is a vector subspace of it for every $1 \leq k \leq m$.
	We expand $\mathcal L_m$ by the collection of unary function symbols $\{\lambda_q\}_{q \in \mathbb Q}$ and interpret $\lambda_q$ as the multiplication by a rational $q$.
	Let $\mathcal L_m'$ be the expansion denoted of $\mathcal L_m$.
	Let $T'_m$ be the extension of the theory $T_m$ adding the sentences $\forall x \ \lambda_q(x)=q \cdot x$ for every $q \in \mathbb Q$.
	
	For every $\mathcal L_m'$-substructure $\mathcal A$ of a model of $T_m'$ whose universe is $A$, we define the equivalence relations $\sim^k_{\mathcal A}$ by $$x \sim^k_{\mathcal A} y \Leftrightarrow y-x \in U_k^{\mathcal A}$$ for $1 \leq k \leq m$.
	We define $\sim^0_{\mathcal A}$ by $x \sim^0_{\mathcal A} y \Leftrightarrow y=x$.
	For $a \in A$ and $0 \leq k \leq m$, we put $$\mathcal C_{\mathcal A}^k(a):=\{b \in A\;|\; b \sim^k_{\mathcal A} a\}.$$
	We omit one of or both of the superscript $k$ and subscript $\mathcal A$ of $\sim^k_{\mathcal A}$ and $\mathcal C_{\mathcal A}^k(a)$ if they are obvious from the context.
	Observe that $\mathcal C_{\mathcal A}^k(0) = U_k^{\mathcal A}$.
	\medskip
	
	
	{\textbf{Claim 1.}} Suppose $A \neq \mathcal C_{\mathcal A}^k(0)$. Let $a \in A$. There exists $b \in A$ such that $b<a$ and $b \not\sim^k_{\mathcal A} a$.
	\begin{proof}[Proof of Claim 1]
		If $a \sim 0$, choose $b<0$ with $b \notin \mathcal C(0)$.
		This is possible because $A$ is a divisible ordered subgroup and $A \neq \mathcal C(0)$.
		Suppose $a \not\sim 0$.
		If $a>0$, put $b=0$.
		If $a<0$, put $b=2a$.
	\end{proof}
	
	{\textbf{Claim 2.}} Let $a_1, a_2 \in A$ with $a_1<a_2$ and $a_1 \not\sim^k_{\mathcal A} a_2$.
	Let $n$ be a integer $n>1$ and $k_i$ be integers with $0 \leq  k_i \leq n$ for $i=1,2$.
	Then, $(k_1a_1+(n-k_1)a_2)/n \not\sim^k_{\mathcal A} (k_2a_1+(n-k_2)a_2)/n$ if $k_1 \neq k_2$.
	\begin{proof}[Proof of Claim 2]
		Suppose $(k_1a_1+(n-k_1)a_2)/n \sim (k_2a_1+(n-k_2)a_2)/n$.
		We have
		\begin{align*}
			(k_1a_1+(n-k_1)a_2)/n - (k_2a_1+(n-k_2)a_2)/n = (k_1-k_2)(a_1-a_2)/n \in \mathcal C(0).
		\end{align*}
		Since $\mathcal C(0)$ is a subgroup, we have $(k_1-k_2)(a_1-a_2) \in \mathcal C(0)$.
		Assume for contradiction $k_1 \neq k_2$.
		Since $\mathcal C(0)$ is convex and $A$ is divisible, we have $a_1-a_2 \in \mathcal C(0)$, which is absurd.
	\end{proof}
	
	We prove that $T_m'$ admits quantifier elimination as an $\mathcal L'_m$-theory.
	
	\begin{proof}[Proof of QE]
		For convenience, put $\mathcal L_0=\{<,+\}$ and let $T_0$ be the collection of axioms of ordered divisible Abelian group.
		We define $\mathcal L'_0$ in the same manner as the definition of $\mathcal L'_m$.
		We prove that $T_m'$ admits quantifier elimination by induction on $m$.
		
		Let $\mathcal N_j$ be a models of $T_m'$ for $j=1,2$.
		Let $\mathcal A$ be their common substructure.
		We denote the universes of $\mathcal N_1$, $\mathcal N_2$ and $\mathcal A$ by $N_1$, $N_2$ and $A$, respectively.
		Let $\phi$ be an arbitrary primitive existential formula with parameters from $A$.
		In other words, there exist finitely many basic formulas $\phi_1(x), \ldots, \phi_l(x)$ with the single free variable $x$ such that $$\phi = \exists x\ \bigwedge_{i=1}^l \phi_i(x).$$
		By a QE test (cf.\ \cite[Theorem 3.2.5]{TZ}), we have only to show that $\mathcal N_2 \models \phi$ whenever $\mathcal N_1 \models \phi$.

		Every basic formula $\phi_i(x)$ for $1 \leq i \leq l$ is equivalent to one of the following forms or their negation:
		\begin{enumerate}
			\item[(a)] $x *_i c_i$, where $c_i \in A$ and $*_i \in \{=, <, >\}$;
			\item[(b)] $U_k(q_ix+a_i)$, where $1 \leq k \leq m$, $q_i \in \mathbb Q$ and $a_i \in A$.
		\end{enumerate}
		Consider the case in which $\phi_i(x)$ is of the form (b) or its negation.
		Put $c_i:=-a_1/q_i \in A$.
		It is easy to show that $\mathcal M \models U_k(q_ix+a_i) \Leftrightarrow  x\in \mathcal C^k_{\mathcal N_j}(c_i)$ for each $x \in M$ and $j=1,2$.
		Put $P:=\{c_i\;|\; 1 \leq i \leq l\}$.
		
		Let $n_1 \in N_1$ such that $\mathcal N_1 \models \bigwedge_{i=1}^l \phi_i(n_1)$.
		We want to find $n_2 \in N_2$ such that $\mathcal N_2 \models \bigwedge_{i=1}^l \phi_i(n_2)$.
		We consider the following three separate cases.
		\medskip
		
		\textbf{Case 1.} Suppose $n_1 \in A$. Put $n_2=n_1$ in this case.
		\medskip
		
		\textbf{Case 2.} Suppose $n_1 \notin A$ and there exists $a \in A$ such that $n_1 \sim^m_{\mathcal M} a$.
		By shifting, we may assume that $a=0$ without loss of generality.
		We may assume that $\phi_1,\ldots, \phi_{l'}$ are $\mathcal L_{m-1}'$-formulas by arranging the order of the formulas $\phi_1,\ldots, \phi_{l}$.
		Observe that $\mathcal C_{\mathcal N_j}^m(0)$ are models of $T_{m-1}$ for $j=1,2$.
		We have $n_1 \in \mathcal C^m_{\mathcal N_1}(0)$.
		By the induction hypothesis and the QE test, there exists an $n_2 \in \mathcal C^m_{\mathcal N_2}(0)$ such that $\mathcal N_2 \models \bigwedge_{i=1}^{l'} \phi_i(n_2)$.
		Since $n_i \in \mathcal C^m_{\mathcal N_i}(0)$ for $i=1,2$, we have $n_1 \in \mathcal C_{\mathcal N_1}^m(a') \Leftrightarrow n_2 \in \mathcal C_{\mathcal N_2}^m(a')$ for $a' \in A$.
		This deduces $\mathcal N_2 \models \bigwedge_{i=l'+1}^{l} \phi_i(n_2)$.
		We have shown $\mathcal N_2 \models \bigwedge_{i=1}^l \phi_i(n_2)$.
		\medskip
		
		\textbf{Case 3.} Suppose $n_1 \not\sim^m_{\mathcal M} a$ for every $a \in A$.
		Enumerate the elements of $P$, say $b_1, \ldots, b_p$, in the increasing order.
		Suppose $m=0$.
		Every $\phi_i$ is equivalent to formulas of the form $x *_i c_i$ with $c_i \in A$ and $*_i \in \{=, <, >\}$.
		Choose $n_2$ as an element of $A$ smaller than $b_1$ if $n_1<b_1$.
		Choose $n_2$ as an element of $A$ larger than $b_p$ if $n_1>b_p$.
		Put $n_2:=(b_{i-1}+b_i)/2$ if there exists $1<i \leq p$ such $b_{i-1}<n_1<b_i$.
		It is a routine to show $\mathcal N_2 \models \bigwedge_{i=1}^l \phi_i(n_2)$.
		We omit the proof.
		
		Suppose $m>0$.
		Observe that $\mathcal C_{\mathcal N_j}^{k_1}(a) \subseteq \mathcal C_{\mathcal N_j}^{k_2}(a)$ for $j=1,2$, $a \in N_j$ and $1 \leq k_1 \leq k_2 \leq m$.
		In particular, the case assumption implies $n_1 \notin \mathcal C_{\mathcal N_1}^k(b_i)$ for $1 \leq i \leq l$ and $1 \leq k \leq m$.
		This implies that $\phi_i(x)$ is not of the form (b) for every $1 \leq i \leq l$ because $\mathcal N_1 \models \bigwedge_{i=1}^k \phi_i(n_1)$.
		Changing the  the index appropriately, we may assume $\phi_i(x)$ is of the form (a) for $1 \leq i \leq q$ and $\phi_i(x)$ is the negations of formulas of the form (b) for $q< i \leq l$.
		
		Put $\psi(x)=\bigwedge_{i=1}^q \phi_i(x) \wedge \bigwedge_{i=1}^p (x \notin \mathcal C^m(b_i))$.
		We have $T_m' \models \forall x\ (\psi(x) \rightarrow \bigwedge_{i=1}^l \phi_i(x))$ and $\mathcal N_1 \models \psi(n_1)$.
		If $n_1<b_1$, there exists $n_2 \in A$ so that $n_2<b_1$ and $n_2 \not\sim^m_{\mathcal A} b_1$ by Claim 1.
		This $n_2$ satisfies $\mathcal N_2 \models \psi(n_2)$, and this $n_2$ is the desired element.
		We can choose the desired element $n_2 \in A$ in the same manner if $n_1>b_k$.
		If $b_i<n_1<b_{i+1}$ for some $1 \leq i <k$, $n_2:=(b_i+b_{i+1})/2$.
		We have $\mathcal N_2 \models \phi_i(n_2)$ for $1 \leq i \leq q$ and $n_2 \notin \mathcal C^m_{\mathcal N_2}(b_i)$ for $1 \leq i \leq p$ by Claim 2.
		Therefore, we have $\mathcal N_2 \models \bigwedge_{i=1}^l \phi_i(n_2)$.
	\end{proof}
	
	From now on, we fix a model $\mathcal M=(M,<,+,U_1,\ldots, U_m)$ of $T_m$, and the subscript $\mathcal M$ of $\sim_{\mathcal M}^k$ and $\mathcal C_{\mathcal M}^k$ is omitted.
	Observe that $\mathcal M$ is naturally a model of $T'_m$.
	Let $a \in M$.
	Put
	\begin{align*}
		&\mathcal D^+_k(a):=\{x \in M\;|\; x>a, x \notin \mathcal C^k(a)\} \text{ and }\\
		&\mathcal D^-_k(a):=\{x \in M\;|\; x<a, x \notin \mathcal C^k(a)\}
	\end{align*}
	for $0 \leq k \leq m$.
	Observe that $\mathcal D^+_0(a):=\{x \in M\;|\; x>a\}$ and $\mathcal D^-_0(a):=\{x \in M\;|\; x<a\}$
	
	Let $X$ be a definable subset of $M^n$ for $n>0$.
	Here, $M^0$ is considered to be a singleton.
	Recall that an \textit{affine function} $f:M^{n-1} \to M$ with coefficients in $\mathbb Q$ is of the form $$f(\overline{x})=\overline{q} \cdot \overline{x} + a,$$ where $\overline{q} \in \mathbb Q^{n-1}$, $a \in M$ and $\overline{q} \cdot \overline{x}$ is the inner product.
	The restriction of an affine function to subsets of $M^{n-1}$ is also called affine by abuse of the terminology.
	We say that $X$ is \textit{good} if there exist a subset $\mathcal I$ of $\{1,2,3\}$, $0 \leq k_i \leq m$ for $i \in \mathcal I $ and affine functions $f_i:\Pi(X) \to M$ with coefficients in $\mathbb Q$ for $i \in \mathcal I$ such that $X$ is of the following form:
	\begin{align*}
		X=\bigcap_{i \in \mathcal I} \mathcal S_i^{k_i}(f_i(x)),
	\end{align*}
	where
	\begin{align*}
		& \mathcal S_1^{k_1}(f_1(x)) =\{(x,y) \in \Pi(X) \times M\;|\; y \in \mathcal C^{k_1}(f_1(x))\}, \\
		& \mathcal S_2^{k_2}(f_2(x)) =\{(x,y) \in \Pi(X) \times M\;|\; y\in \mathcal D^-_{k_2}(f_2(x))\} \text { and } \\
		& \mathcal S_3^{k_3}(f_3(x)) =\{(x,y) \in \Pi(X) \times M\;|\; y \in \mathcal D^+_{k_3}(f_3(x))\}. 
	\end{align*}
	We put $\mathcal S_i^{k_i}(f_i(x)):=\Pi(M) \times M$ for $1 \leq i \leq 3$ with $i \notin \mathcal I$ in order to avoid annoying case distinction.
	We have $X=\bigcap_{i=1}^3 \mathcal S_i^{k_i}(f_i(x))$ by using this new notation.
	Note that, if $n=1$, the functions $f_i(x)$ are constants.
	We often omit superscript $k_i$ of $\mathcal S_i^{k_i}(f_i(x))$.
	\medskip
	
	\textbf{Claim 3.} Every definable subset $X$ of $M^n$ is a union of finitely many good definable subsets.
	\begin{proof}[Proof of Claim 3]
		We say that a formula $\phi(\overline{x},y)$ with the tuple of $(n-1)$ free variables $\overline{x}$, the free variable $y$ and parameters from $M$ is \textit{primitive} if there exist nonnegative integers $l_i$ for $1 \leq i \leq 3$ and a formula $\psi(\overline{x})$ with parameters from $M$ such that 
		\begin{align*}
			\phi(\overline{x},y) &=\psi(\overline{x}) \wedge \bigwedge_{i=1}^{l_1}(y \in \mathcal C^{k_{1i}}(g_{1i}(\overline{x}))) \wedge \bigwedge_{i=1}^{l_2}(y \in \mathcal D^+_{k_{2i}}(g_{2i}(\overline{x}))) \\
			&\qquad \wedge \bigwedge_{i=1}^{l_3}(y \in \mathcal D^-_{k_{3i}}(g_{3i}(\overline{x}))),
		\end{align*}
		where $l_j \geq 0$, $g_{ji}(\overline{x})$ is affine functions with coefficients in $\mathbb Q$ and $0 \leq k_{ji} \leq m$ for $1 \leq j \leq 3$ and $1 \leq i \leq l_j$.
		Here, the notation $\bigwedge_{i=1}^0 (\text{formula})$ means the tautology $\top$.
	
		Since $T$ admits quantifier elimination as an $\mathcal L'$-theory, $X$ is a union of finitely many definable sets defined by primitive formulas.
		Therefore, we may assume that $X$ is defined by a primitive formula $\phi(\overline{x},y)$.
		
		Suppose $l_1 \neq 0$.
		Put $\widetilde{k}:=\min\{k_{1i}\;|\; 1 \leq i \leq l_1\}$ and choose $1 \leq \widetilde{i} \leq l_1$ such that $\widetilde{k}=k_{1\widetilde{i}}$.
		By the definition of the equivalence class $\mathcal C^k(a)$, we get $$\bigwedge_{i=1}^{l_1}(y \in \mathcal C^{k_{1i}}(g_{1i}(\overline{x}))) \leftrightarrow (y \in \mathcal C^{\widetilde{k}}(g_{1\widetilde{i}}(\overline{x})))$$ for every $(\overline{x},y) \in M^n$ with $\mathcal M \models (\exists y\ \bigwedge_{i=1}^{l_1}(y \in \mathcal C^{k_{1i}}(g_{1i}(\overline{x}))))$.
		Therefore, we may assume that $l_1=0$ or $l_1=1$ without loss of generality.
		
		We may assume $l_j=0$ or $l_j=1$ for $j=2,3$.
		We only discuss about the case in which $j=2$, but the case in which $j=3$ is similar. 
		
		Suppose $l_2>0$.
		Put $\mathcal P_k:=\{1 \leq i \leq l_2\;|\; k_{2i}=k\}$ for $0 \leq k \leq m$ and $\mathcal K:=\{0 \leq k \leq m\;|\; \mathcal P_k \neq \emptyset\}$.
		By partitioning $\Pi(X)$ into finitely many definable sets, we can reduce to the case where, for each $k \in \mathcal K$, there exists $p_k \in \mathcal P_k$ such that $g_{2p_k}(\overline{x}) = \sup \{g_{2i}(\overline{x})\;|\; i \in \mathcal P_k\}$ for each $\overline{x} \in \Pi(X)$.
		Put $h_k(\overline{x}):=g_{2p_k}(\overline{x})$.
		We have $$\bigwedge_{i \in \mathcal P_k}(y \in \mathcal D^+_{k_{2i}}(g_{2i}(\overline{x}))) \leftrightarrow y \in \mathcal D^+_{k}(h_k(\overline{x}))$$ for $\overline{x} \in \Pi(X)$ and $k \in \mathcal K$.
		
		By partitioning $\Pi(X)$ into finitely many definable sets again, we may assume that, for every $k_1,k_2 \in \mathcal K$ and $0 \leq k \leq m$, there exist $\square_{1,k_2,k_2} \in \{<,=,>\}$ and $\square_{2,k_1,k_2,k} \in \{\sim^k, \not\sim^k\}$ such that the relations 
		\begin{center}
		$h_{k_1}(\overline{x}) \ \square_{1,k_2,k_2}\  h_{k_2}(\overline{x})$ and $h_{k_1}(\overline{x}) \ \square_{2,k_2,k_2,k}\  h_{k_2}(\overline{x})$
		\end{center}
		holds for every $\overline{x} \in \Pi(X)$.
		Let $k' \in \mathcal K$ be such that $h_{k'}(\overline{x})=\sup\{h_k(\overline{x})\;|\;k \in \mathcal K\}$ for every $\overline{x} \in \Pi(X)$.
		Let $\widetilde{k}:=\sup\{k \in \mathcal K\;|\; h_k(\overline{x}) \sim_{\mathcal M}^k h_{k'}(\overline{x})\}$.
		Then we have $$\bigwedge_{k \in\mathcal K} (y \in \mathcal D^+_{k}(h_k(\overline{x}))) \leftrightarrow y \in \mathcal D^+_{\widetilde{k}}(h_{\widetilde{k}}(\overline{x}))$$ for $\overline{x} \in \Pi(X)$.
		In summary, we have $$\bigwedge_{i=1}^{l_2}(y \in \mathcal D^+_{k_{2i}}(g_{2i}(\overline{x}))) \leftrightarrow y \in \mathcal D^+_{\widetilde{k}}(h_{\widetilde{k}}(\overline{x}))$$
		for $\overline{x} \in \Pi(X)$.
		We have reduced to the case in which $l_2=1$.
	\end{proof}
	
	We show that $\mathcal M$ is a weakly o-minimal structure.
	Let $X$ be a definable subset of $M$.
	We show that $X$ is a union of a finite set and a finitely many open convex sets.
	$X$ is a finite union of good definable sets by Claim 3.
	Therefore, we may assume that $X$ is a good definable set without loss of generality.
	We define $\mathcal I$, $k_i$, $f_i(x)$ and $\mathcal S_i^{k_i}(f_i(x))$ as in the definition of good definable sets.
	$\mathcal S_i^{k_i}(f_i)$ is a singleton if $i=1$ and $k_1=0$, and $\mathcal S_i^{k_i}(f_i)$ is an open convex set in the other cases.
	Observe that the intersection of two open convex sets is open and convex.
	Therefore, $X$ is either a singleton or an open convex set.
	We have shown that  $\mathcal M$ is weakly o-minimal.
	\medskip
	
	We prove that the dimension function {$\topdim$} is a dimension function.
	For the simplicity of notations, we write $\dim$ instead of $\topdim$.
	By the definition of the topological dimension, conditions $(1)$ and $(3)_n$ in Definition \ref{def:dimension} hold for every $n$.
	Condition $(2)_n$ holds for every $n$ by \cite{A} and \cite[Corollary 4.4]{MMS}.
	Recall that $\overline{M}$ denotes the definable Dedekind completion of $M$ under the order $<$.
	According to \cite[Theorem 4.2]{Wencel}, in weakly o-minimal structures, satisfying condition $(4)_n$ for all $n>1$ is equivalent to the condition that, for every definable function $f:I \to \overline{M}$ from a nonempty open interval $I$, there exists a nonempty interval $J$ such that the restriction of $f$ to $J$ is continuous. 
	By the definition of definable continuous functions into the Dedekind completion, there exists a definable subset $X$ of $M^2$ such that $f(x)=\sup X_x^{\Pi_1^2}$ for $x \in I$.
	
	Let $X=\bigcup_{i=1}^lX_i$ be a partition of $X$ into good definable sets given in Claim 3.
	At least one of $\Pi_1^2(X_i)$ contains a nonempty open interval $I'$ by $(2)_1$.
	We may assume that $X$ is a good definable set by considering $X_i \cap (\Pi_1^2)^{-1}(I')$ instead of $X$.
	
	We define $\mathcal I$, $k_i$, $f_i(x)$ and $\mathcal S_i^{k_i}(f_i(x))$ as in the definition of good definable sets.
	Put $h_1(x)=\sup \{f_1(x)+t\;|\; t \in U_{k_1}^{\mathcal M}\}$ and $h_2(x):=\inf\{f_2(x) + t\;|\; t \in U_{k_2}^{\mathcal M}\}$.
	Since $X$ is bounded above, at least one of $1$ and $2$ belongs to $\mathcal I$.
	We have $f(x) =\sup\{h_i(x)\;|\; i \in \{1,2\} \cap \mathcal I\}$, which is continuous.
	We have shown that {$\topdim$} is a dimension function.
	\medskip

	For $1 \leq  \kappa \leq m$, we construct a dimension function $\mydims^{\kappa}:\myDef(\mathcal M) \to \mathbb N \cup \{-\infty\}$ such that $\mydims^{\kappa}(U_{\kappa}^{\mathcal M})=0$ and $\mydims^{\kappa}(U_{\kappa+1}^{\mathcal M})=1$.
	For simplicity of notations, we write $\mydims$, $\mathcal C(\cdot)$ and $\sim$ instead of $\mydims^{\kappa}$, $\mathcal C_{\mathcal M}^{\kappa}(\cdot)$ and $\sim_{\mathcal M}^{\kappa}$.
	
	We put $\mydims(\emptyset)=-\infty$, and we define $\mydims(X)$ for each nonempty definable subset $X$ of $M^n$ by induction on $n$.
	
	\begin{itemize}
		\item If $n=1$. We put $\mydims X=0$ if and only if there exist finitely many points $a_1,\ldots, a_l \in M$ such that $X \subseteq \bigcup_{i=1}^l\mathcal C(a_i)$.
		Otherwise, we put $\mydims(X)=1$.
		\item If $n>1$, put $X(i)=\{x \in M^{n-1}\;|\; \mydims(X_x^{\Pi})=i\}$ for $i=0,1$.
		We will prove that both $X(0)$ and $X(1)$ are definable later.
		Put $$\mydims(X):=\max\{\mydims(X(1))+1,\mydims(X(0))\}.$$
	\end{itemize}
	
	We first prove the following:
	\medskip
	
	\textbf{Claim 4.} For every definable convex subset $I$ of $M$, $\mydims(I)=0$ if and only if $I$ is contained in $\mathcal C(x)$ for some $x \in I$.
	\begin{proof}[Proof of Claim 4]
		The `if' part is obvious from the definition of $\mydims$.
		The `only if' part is a direct corollary of Claim 3.
	\end{proof}
	
	\textbf{Claim 5.} Let $1 \leq k_1<k_2 \leq m$ and $\kappa<k_2$.
	Suppose $\mathcal D_{k_1}^+(a_1) \cap \mathcal C^{k_2}(a_2) \neq \emptyset$, then $\mydims(\mathcal D_{k_1}^+(a_1) \cap \mathcal C^{k_2}(a_2))=1$.
	\begin{proof}[Proof of Claim 5]
		Put $X:=\mathcal D_{k_1}^+(a_1) \cap \mathcal C^{k_2}(a_2)$.
		Pick $b \in X$.
		Let $0<t \in \mathcal C^{k_2}(0) \setminus \mathcal C^{\kappa}(0)$.
		We have $b, b+t \in X$ and $b \not\sim b+t$.
		By Claim 4, we have $\mydims(\mathcal D_{k_1}^+(a_1) \cap \mathcal C^{k_2}(a_2))=1$.
	\end{proof}
	
	\textbf{Claim 6.} Let $f: M \to M$ be an affine function with coefficient in $\mathbb Q$.
	\begin{enumerate}
		\item[(a)] Let $a,b \in M$.
		If $a \sim b$, then $f(a) \sim f(b)$.
		\item[(b)] Let $a,b \in M$. Suppose $f$ is not constant.
		If $a \not\sim b$, then $f(a) \not\sim f(b)$.
		\item[(c)] Put $\mathcal C(S):=\bigcup_{a \in S}\mathcal C(a)$ for a subset $S$ of $M$.
		Then, $\mathcal C(f(a))=\mathcal C(f(\mathcal C(a)))$.
	\end{enumerate}
	\medskip
	
	The proof of Claim 6 is straightforward and left to readers.
	\medskip
	
	\textbf{Claim 7.} Let $X$ be a good definable subset of $M$.
	That is, there exist $\mathcal I \subseteq \{1,2,3\}$, nonnegative integers $k_i$ and $a_i \in M$ for $i \in \mathcal I$ such that $X=\bigcap_{i \in \mathcal I} \mathcal S_i^{k_i}(a_i)$.
	Then, $\mydims(X)=0$ if and only if one of the following conditions holds:
	\begin{enumerate}
		\item[(a)] $1 \in \mathcal I$ and $k_1 \leq \kappa$;
		\item[(b)] $1 \notin \mathcal I$, $2,3 \in \mathcal I$, $k_2 \leq \kappa$, $k_3 \leq \kappa$, and $a_2 \sim a_3$;
		\item[(c)] $1,2,3 \in \mathcal I$, $k_1 > \kappa$, $k_2 \leq \kappa$, $k_3 \leq \kappa$, $a_1 \sim^{k_1} a_2$, $a_1 \sim^{k_1} a_3$ and $a_2 \sim a_3$.
	\end{enumerate}
	Furthermore, there exists $i \in \{1,2\}$ such that $X \subseteq \mathcal C(a_i)$ if $\mydims(X)=0$. 
	\begin{proof}[Proof of Claim 7]
		In this proof, we omit the superscript $k_i$ of $\mathcal S_i^{k_i}(a_i)$.
		If $1 \in \mathcal I$ and $k_1 \leq \kappa$, we have $\mydims(X)=0$ by Claim 4.
		
		Consider the case in which $2 \notin \mathcal I$ or $3 \notin \mathcal I$.
		We only consider the case in which $2 \notin \mathcal I$ because the proof of the other case is similar.
		If $1 \notin \mathcal I$, $X$ is unbounded and we have $\mydims(X)=1$ by Claim 1.
		If $1 \in \mathcal I$, $k_1>\kappa$, then $\mydims(X)=1$ by Claim 5.
		
		Consider the remaining case in which $2 \in \mathcal I$ and $3 \in \mathcal I$.
		Suppose $1 \notin \mathcal I$.
		If $k_2 \geq \kappa$ and $a_2 \sim a_3$, we have $\mathcal  S_2(a_2) \cap \mathcal S_3(a_3)=\emptyset$.
		The case $k_3 \geq \kappa$ is similar.
		We may assume that $k_2<\kappa$ and $k_3<\kappa$.
		If $a_2 \not\sim a_3$, we have $(a_2+2a_3)/3, (2a_2+a_3)/3 \in \mathcal S_2(a_2) \cap \mathcal S_3(a_3)$ and $(a_2+2a_3)/3 \not\sim (2a_2+a_3)/3$ by Claim 2.
		This implies $\mydims(X)=1$ by Claim 4.
		If $a_2 \sim a_3$, we have $\mathcal S_2(a_2) \cap \mathcal S_3(a_3) \subseteq \mathcal C(a_2)$ and we get $\mydims(X)=0$ by Claim 4.
	
		Suppose $1 \in \mathcal I$ and $k_1>\kappa$.
		Consider the case where $\mathcal S_2(a_2) \cap \mathcal S_3(a_3) \not\subseteq S_1(a_1)$.
		We have either $X=\mathcal S_1(a_1) \cap \mathcal S_3(a_3)$ or $X=\mathcal S_1(a_1) \cap \mathcal S_2(a_2)$.
		Let us consider the case $X=\mathcal S_1(a_1) \cap \mathcal S_3(a_3)$.
		We have $\mydims (X)=1$ by Claim 5 in this case.
		We get $\mydims (X)=1$ if $X=S_1(a_1) \cap \mathcal S_2(a_2)$ similarly.
		
		The next case is the case where $\mathcal S_2(a_2) \cap \mathcal S_3(a_3) \subseteq \mathcal S_1(a_1)$.
		We have $\mathcal S_2(a_2) \cap \mathcal S_3(a_3)=\emptyset$ unless $k_2<k_1$ and $k_3<k_1$.
		Therefore, we assume $k_2<k_1$ and $k_3<k_1$.
		The inclusion $\mathcal S_2(a_2) \cap S_3(a_3) \subseteq \mathcal S_1(a_1)$ implies $a_1 \sim^{k_1} a_2$ and $a_1 \sim^{k_1} a_3$.
		In this case, the equality $X=\mathcal S_2(a_2) \cap \mathcal S_3(a_3)$ holds.
		We have $\mydims(X)=0$ if and only if $a_2 \sim a_3$ by Claim 4.
		We have $X \subseteq \mathcal C(a_2)$ in this case.
		In summary, we get $\mydims(X)=0$ if and only if $k_2 \leq \kappa$, $k_3 \leq \kappa$, $a_1 \sim^{k_1} a_2$, $a_1 \sim^{k_1} a_3$ and $a_2 \sim a_3$
	\end{proof}

	We show that $\mydims$ is a dimension function.
	It is easy to show that $\mydims$ satisfies conditions (1) and $(2)_1$ of Definition \ref{def:dimension}.
	We omit the proof.
	
	We show that $X(0)$ and $X(1)$ definable.
	We may reduce to simpler cases by covering $X$ by finitely many simpler definable sets $X_1,\ldots X_l$.
	In fact, we have $X(1)=\bigcup_{i=1}^l X_i(1)$ and $X(0)=\Pi(X) \setminus X(1)$ by $(2)_1$.
	If $X_i(0)$ and $X_i(1)$ are definable for all $1 \leq i \leq l$, then $X(0)$ and $X(1)$ are definable.
	By Claim 3, we may assume that $X$ is a good definable set without loss of generality.
	Claim 7 implies that $X(0)$ and $X(1)$ are definable. 
	Condition $(4')_n$ in Definition \ref{def:dimension2} obviously holds by the definition of $\mydims$.
	
	We show that $\mydims$ satisfies condition $(3')$ in Definition \ref{def:dimension2}.
	Let $X$ be a definable subset of $M^2$.
	By Lemma \ref{lem:2n}, condition $(2)_2$ in Definition \ref{def:dimension} holds.
	Using $(2)_2$ and Claim 3, we may assume that $X$ is a good definable set without loss of generality.
	We define $\mathcal I$, $k_i$, $f_i(x)$ and $\mathcal S_i^{k_i}(f_i(x))$ as in the definition of good definable sets.
	Since $\mathcal M$ is weakly o-minimal, $\Pi(X)$ is a union of a finite set and finitely many definable open convex set.
	By partitioning $\Pi(X)$, by $(2)_2$, we may assume the following: 
	\begin{itemize}
		\item $\mydims(X_x^{\Pi})$ is independent of $x \in \Pi(X)$;
		\item For every $i,j \in \mathcal I$ and $1 \leq k \leq m$, there exists $\square_{i,j,k} \in \{\sim_k, \not\sim_k\}$ such that $f_i(x) \square_{i,j,k} f_j(x)$ holds for every $x \in \Pi(X)$;
		\item For every $i,j \in \mathcal I$, there exists $\square_{i,j} \in \{=,<,>\}$ such that  $f_i(x) \square_{i,j} f_j(x)$ for every $x \in \Pi(X)$;
		\item $\Pi(X)$ is a singleton or an open convex set.
	\end{itemize}
	Put $Y:=X^{\text{switch}}$.
	
	First consider the case in which $\mydims(\Pi(X))=0$.
	Suppose that there exists $x \in \Pi(X)$ such that $\mydims(X_x^{\Pi})=1$.
	We have $\mydims(X)=1$ by the definition of $\mydims$.
	We have $X_x^{\Pi} \subseteq \Pi_1^2(Y)$, which implies $\dim \Pi(Y)=1$ by $(2)_1$.
	For every $y \in \Pi(Y)$, we have $Y_y^{\Pi} \subseteq \Pi(X)$, which implies $\mydims(Y_y^{\Pi})=0$.
	We get $\mydims(Y)=1=\mydims(X)$.
	
	Suppose that $\mydims(X_x^{\Pi})=0$ for every $x \in \Pi(x)$.
	We have $\mydims(X)=0$ in this case.
	There exist finitely many $a_1,\ldots, a_k \in \Pi(X)$ such that $\Pi(X) \subseteq \bigcup_{i=1}^k \mathcal C(a_i)$ because $\mydims(\Pi(X))=0$.
	In fact, we have $k=1$ by Claim 4 because $\Pi(X)$ is either a singleton or a convex open set.
	By Claim 7, we have $\dim (X_x^{\Pi_1^2})=0$ if and only if $X_x^{\Pi_1^2} \subseteq \mathcal C(f_i(x))$ for some $i=1,2$.
	For $x \in \Pi(X)$, by Claim 6(a), we have $f_i(x) \sim f_i(a_1)$ for $i \in \mathcal I$ because $x \in \mathcal C(a_1)$.
	Therefore, we have $f_i(x) \in \mathcal C(f_i(a_1))$ for $x \in \Pi(X)$.
	We have shown that $\Pi(Y) \subseteq \bigcup_{i \in \mathcal I} \mathcal C(f_i(a_1))$.
	This implies $\mydims(\Pi(Y))=0$.
	For every $y \in \Pi(Y)$, we have $Y_y^{\Pi} \subseteq \Pi(X)$, which implies $\mydims(Y_y^{\Pi})=0$.
	We get $\mydims Y = 0 = \mydims X$.
	
	Next consider the case in which $\mydims(\Pi(X))=1$.
	Suppose that $\mydims(X_x^{\Pi})=0$ for every $x \in \Pi(x)$.
	We have $\mydims X=1$.
	There exists $i \in \{1,2\}$ such that $X_x^{\Pi_1^2} \subseteq \mathcal C(f_i(x))$ by Claim 7.
	If $f_i(x)$ is a constant function whose value is constantly $c$, we obtain $\Pi(Y) \subseteq \mathcal C(c)$.
	This means $\mydims (\Pi(Y))=0$.
	Assume for contradiction that $\mydims (Y_y^{\Pi})=0$ for every $y \in \Pi(Y)$.
	We have $\mydims (Y^{\text{switch}})=\mydims Y=0$ by the previous argument.
	However,  $Y^{\text{switch}}=X$ and we get $\mydims X=0$, which is absurd.
	We have shown that there exists $y \in \Pi(Y)$ with $\mydims(Y_y^{\Pi})=1$.
	We get $\mydims Y=1=\mydims X$.
	
	Next assume that $f_i$ is not a constant function.
	Since $f_i$ is affine, the inverse $f_i^{-1}$ exists.
	By Claim 6(a-c), we have $$\Pi^{-1}(\mathcal C(x)) \cap Y \subseteq \mathcal C(x) \times \mathcal C(f_i^{-1}(x))$$ for every $x \in \Pi(Y)$.
	This yields $\mydims(Y_x^{\Pi}) \leq \mydims( \mathcal C(f_i^{-1}(x)))=0$.
	
	Since $\mydims (\Pi(X))=1$, we can choose infinitely many $a_1,\ldots \in \Pi(X)$ such that $a_{j_1} \not\sim a_{j_2}$ whenever $j_1 \neq j_2$.
	Since $X_{a_j}^{\Pi}$ is not empty and contained in $\mathcal C(f_i(a_j))$, we have $\mathcal C(f_i(a_j)) \cap \Pi(Y) \neq \emptyset$ for every $j$.
	Choose $b_j \in \mathcal C(f_i(a_j)) \cap \Pi(Y)$ for every $j$.
	By Claim 6(b), we get $b_{j_1} \not\sim_{\mathcal M} b_{j_2}$ whenever $j_1 \neq j_2$.
	This means $\mydims (\Pi(Y))=1$.
	We have shown $\mydims(Y)=1=\mydims(X)$.
	
	Let us consider the case in which $\mydims (\Pi(X))=1$ and $\mydims(X_x^{\Pi})=1$ for every $x \in \Pi(X)$.
	In this case, $\Pi(X)$ is a convex open set.
	We put $S_1^{m+1}(a)=M$ for $a \in M$.
	By Claim 7, for every $x \in \Pi(X)$, one of the following holds:
	\begin{enumerate}
		\item[(a)] $X_x^{\Pi}=\mathcal S_1^{k_1}(f_1(x))$ with $k_1 > \kappa$.
		\item[(b)] $k_1>\kappa$ and there exists $i=2,3$ such that $X_x^{\Pi}=\mathcal S_i^{k_i}(f_i(x)) \cap \mathcal S_1^{k_1}(f_1(x))$ for every $x \in \Pi(X)$;
		\item[(c)] $2,3 \in \mathcal I$, $f_2(x) \not\sim f_3(x)$ and $X_x^{\Pi}=\bigcap_{i=2}^3\mathcal S_i^{k_i}(f_i(x))$ for every $x \in \Pi(X)$.
	\end{enumerate}
	For case (a), we have $Y=\mathcal C^{k_1}(c) \times \Pi(X)$ if $f_1$ is constantly equal to $c \in M$.
	In the other case, we have $\mydims \Pi(Y)=1$ by Claim 4 and Claim 6(b), and $Y_x^{\Pi}=\mathcal C^{k_1}(f_1^{-1}(x))$ for $x \in \Pi(Y)$ by Claim 6(a-c).
	We have $\dim Y=2$ in this case.
	
	The proofs for cases (b) and (c) are similar.
	We consider case (c) only.
	
	Observe that $f_3(x)<f_2(x)$ and, by Claim 4, we have $f_2(x) \not\sim f_3(x)$ for every $x \in \Pi(X)$.
	Put $g_k(x)=(kf_3(x)+(5-k)f_2(x))/5$ for $1 \leq k \leq 4$.
	Pick $a \in \Pi(X)$.
	At least one of the equalities $\mydims((a,\infty) \cap \Pi(X))=1$ and  $\mydims((-\infty,a) \cap \Pi(X))=1$ holds by $(2)_1$.
	We may assume $\mydims((a,\infty) \cap \Pi(X))=1$ without loss of generality.
	We want to find $b_2 \in \Pi(X) \cup \{+\infty\}$ so that $a<b_2$, $a \not\sim b_2$ and $g_2(a)>g_1(x)$ for all $a<x<b_2$.
	
	First, suppose that there exists $b \in \Pi(X)$ such that $b>a$, $g_1(b)=g_2(a)$.
	If $b \sim a$, then $g_1(a) \sim g_1(b) \sim g_2(a)$ by Claim 6(a).
	However, this relation contradicts Claim 2.
	We have shown that $b \not\sim a$.
	Note that the intermediate value property holds for affine functions in $\mathcal M$.
	The $b_2:=b$ suffices our requirements in this case.
	Second, suppose such $b$ doesn't exists.
	We have $g_1(x)<g_2(a)$ for every $x \in \Pi(X) \cap (a,\infty)$.
	Put $b_2=+\infty$ in this case.
	
	In the same manner, we choose $b_3 \in \Pi(X) \cup \{+\infty\}$ so that $a<b_3$, $a \not\sim b_3$ and $g_3(a)<g_4(x)$ for all $a<x<b_3$.
	Let $b=\min\{b_1,b_3\}$.
	Since $\pi(X)$ is convex, we have $(a,b) \subseteq \Pi(X)$.
	By the choice of $b$, the open box $T:=(a,b) \times (g_2(a),g_3(a))$ is contained in $X$.
	We have $\mydims ((a,b)) =1$ and $\mydims ((g_2(a),g_3(a)))=1$.
	Therefore $2 \geq \mydims Y \geq \mydims (T^{\text{switch}}) =2$ by $(2)_2$.
	
	We have proved that $\mydims$ satisfies $(3')$ in Definition \ref{def:dimension2}.
	Proposition \ref{prop:weak} implies that $\mydims$ is a dimension function.
	\medskip
	
	The final task is to show that there exists no dimension functions on $\mathcal M$  other than the topological dimension and $\mydims^{\kappa}$.
	
	Let $\dim$ be a dimension function on $\mathcal M$.
	First suppose $\dim Z=0$ if and only if $Z$ is a finite set for every nonempty definable subset $Z$ of $M$.
	$\dim$ coincides with the topological dimension on $\myDef_1(\mathcal M)$.
	We have $\dim=\topdim$ by Proposition \ref{prop:dim}(1).
	
	Next suppose there exists an infinite definable subset $Z$ of $\mathcal M$ such that $\dim Z=0$.
	Since $\mathcal M$ is weakly o-minimal, there exists an open convex set $Y$ of $M$ such that $\dim (Y)=0$ and $Y \subseteq X$ by $(2)_1$ of Definition \ref{def:dimension}.
	By shifting $Y$, we may assume that $0 \in Y$ by Proposition \ref{prop:dim}(2).
	Take $b \in Y$ with $b>0$.
	We have $\dim((0,b))=0$ by $(2)_1$.
	
	Put $$\mathcal X=\{x \in M\;|\; \dim((0,x))=0\}.$$
	$\mathcal X$ is a nonempty convex subset of $(0,\infty)$.
	In fact, we have $a_1 \in \mathcal X$ whenever $0<a_1<a_2$ and $a_2 \in \mathcal X$.
	By Lemma \ref{lem:unbounded0}, we have $$\dim \mathcal X=0.$$
	In addition, if $x_1,x_2 \in \mathcal X$, we have $x_1+x_2 \in \mathcal X$.
	In fact, $\dim((x_1,x_1+x_2))=0$ by Proposition \ref{prop:dim}(2).
	We get 
	\begin{align*}
		\dim((0,x_1+x_2))=\max\{\dim((0,x_1)),\dim\{x_1\},\dim((x_1,x_1+x_2))\}=0
	\end{align*}
	by $(2)_1$ of Definition \ref{def:dimension}.
	$\mathcal X$ is bounded.
	Otherwise, we have $\mathcal X=(0,\infty)$ and we have $\dim M=0$, which contradicts Definition \ref{def:dimension}(1).
	If $s:=\sup \mathcal X$ exists in $M$, take an element $t \in \mathcal X$ sufficiently close to $s$.
	We have $2t>s$ and $2t \in \mathcal X$, which is absurd. 
	Therefore $\sup \mathcal X$ does not exist in $M$.
	
	By Claim 3, $\mathcal X$ is a union of finitely many good definable sets $Z_1,\ldots, Z_l$.
	Recall that good definable sets are convex.
	Let $Z$ be the rightmost good definable set among $Z_1,\ldots, Z_l$.
	Let $M_{>a}$ be the set given by $\{x\in M\;|\;x>a\}$.
	Since $\sup \mathcal X$ does not exists in $M$, by the definition of good definable sets, there exist $a \in Z$, $b \in M$ and $1 \leq \kappa \leq m$ such that $Z \cap M_{>a}$ is one of the following form:
	\begin{enumerate}
		\item[(i)] $\mathcal C^{\kappa}(b) \cap M_{>a}$;
		\item[(ii)] $\mathcal D^-_{\kappa}(b) \cap M_{>a}$:
	\end{enumerate}
	We want to show that the feasible case is only case (i) and $b \in \mathcal C^{\kappa}(0)$.
	
	Consider case (ii).
	We have $b>0$ and $b \notin \mathcal C^{\kappa}(0)$ because $\emptyset \neq \mathcal X \subseteq M_{>0}$.
	In particular, we get $b \not\sim^{\kappa} 2b$.
	We get $\mathcal X=\mathcal D^-_{\kappa}(b) \cap M_{>0}$
	On the other hand, we have
	$$\{x_1+x_2\;|\;x_i \in \mathcal D^-_{\kappa}(b) \cap M_{>0} \text{ for }i=1,2 \}= \mathcal D^-_{\kappa}(2b) \cap M_{>0}$$
	Since $\mathcal X$ is closed under addition, we have $\mathcal D^-_{\kappa}(2b) \cap M_{>0} \subseteq \mathcal X$.
	However, $3b/2 \in (\mathcal D^-_{\kappa}(2b) \cap M_{>0}) \setminus \mathcal X$ by Claim 2, which is absurd.
	
	Consider case (i).
	Suppose $b \notin \mathcal C^{\kappa}(0)$.
	Observe that $b>0$.
	We get $b \not\sim^{\kappa} 2b$.
	In particular, $2b \notin \mathcal X$ because $2b \notin \mathcal C^{\kappa}(b)$.
	Observe that $b \in \mathcal C^{\kappa}(b) \cap M_{>0} \subseteq \mathcal X$.
	Since $\mathcal X$ is closed under addition, $2b \in \mathcal X$, which is absurd.

	We have shown that the feasible case is case (i) and $b \in \mathcal C^{\kappa}(0)$.
	In this case, we have
	\begin{equation}
		\mathcal X=\mathcal C^{\kappa}(0) \cap M_{>0}. \tag{***} \label{eq:3rd}
	\end{equation}
	
	We want to show that $\dim$ coincides with $\mydims^{\kappa}$.
	By Proposition \ref{prop:dim}(1) and the definition of $\mydims^{\kappa}$, we have only to show the following claim:
	\medskip
	
	\textbf{Claim 8.} For every definable subset $Y$ of $M$, $\dim Y=0$ if and only if there exist finitely many points $a_1,\ldots,a_l$ such that $Y \subseteq \bigcup_{i=1}^l \mathcal C^{\kappa}(a_i)$.
	\begin{proof}[Proof of Claim 8]
		We get $\dim (\mathcal C^{\kappa}(0))=0$ by equality (\ref{eq:3rd}) and $\dim(\mathcal C^{\kappa}(a))=0$ for every $a \in M$ by Proposition \ref{prop:dim}(2).
		By Definition \ref{def:dimension}(2), we have $\dim(X)=0$ for every nonempty definable subset $X$ of $\mathcal C^{\kappa}(a)$.
		The `if' part is easily deduced from these facts.
		
		We prove the `only if' part.
		Let $Y$ be a definable subset of $M$ with $\dim Y=0$.
		By Claim 3 and Definition \ref{def:dimension}$(2)_1$, we may assume that $Y$ is a good definable set.
		We consider the separate four cases.
		\medskip
		
		\textbf{Case (i).}	Suppose $Y$ is of the form $\mathcal C^k(a)$.
		By Proposition \ref{prop:dim}(2), we get $\dim \mathcal C^k(0)=0$.
		By equality (\ref{eq:3rd}), we have $\mathcal C^k(0) \subseteq \mathcal C^{\kappa}(0)$, which implies $\kappa \geq k$.
		We have $Y \subseteq\mathcal  C^{\kappa}(a)$ in this case.
		\medskip
		
		\textbf{Case (ii).}	Let us consider the case in which $Y$ is of the form $\mathcal C^{k_1}(a_1) \cap \mathcal D_{k_2}^+(a_2)$ and $\mathcal C^{k_1}(a_1) \cap \mathcal D_{k_2}^+(a_2) \neq \mathcal C^{k_1}(a_1)$.
		By taking an element of $Y$ in place of $a_1$, we may assume that $a_1 \in Y$.
		We want to show that $k_1 \leq \kappa$.
		Assume for contradiction that $k_1 > \kappa$.
		If $a_2 \notin \mathcal C^{k_1}(a_1)$, $\mathcal C^{k_1}(a_1) \cap \mathcal D_{k_2}^+(a_2)$ coincides with either an empty set or $\mathcal C^{k_1}(a_1)$ depending on whether $a_1$ is contained in $\mathcal C^{k_2}(a_2)$ or not, which is absurd.
		
		Next suppose $a_2 \in \mathcal C^{k_1}(a_1)$, that is, $a_1 \sim^{k_1} a_2$.
		Since $Y \neq \emptyset$, $k_2<k_1$.
		If $k_2 > \kappa$, choose $t \in \mathcal C^{k_1}(a_1) \cap \mathcal D_{k_2}^+(a_2)$.
		$\mathcal C^{k_2}(t)$ is contained in $Y$.
		We have $\dim(\mathcal C^{k_2}(t))=0$.
		Therefore, by shifting, we have $\mydim (\mathcal C^{k_2}(0))=0$ by Proposition \ref{prop:dim}(2).
		This contradicts the definition of $\mathcal X$ and equality (\ref{eq:3rd}).
		
		Suppose $k_2 \leq \kappa$.
		Put $Z:=Y \cup \mathcal C^{\kappa}(a_2)=(\mathcal C^{k_1}(a_1) \cap \mathcal D_{k_2}^+(a_2)) \cup \mathcal C^{\kappa}(a_2)$.
		Observe that $Z$ is convex.
		Since $\dim(\mathcal C^{\kappa}(a_2))=\dim(\mathcal C^{\kappa}(0))=0$ by Proposition \ref{prop:dim}(2), we have $\dim Z=\dim Y$ by Definition \ref{def:dimension}$(2)_1$.
		Since $a_1 \sim^{k_1} a_2$, we can take $c \in \mathcal C^{k_1}(a_1) \setminus \mathcal C^{\kappa}(a_2)=\mathcal C^{k_1}(a_2) \setminus \mathcal C^{\kappa}(a_2)$ with $c>a_2$.
		By equality (\ref{eq:3rd}) and $c \not\sim^{\kappa} a_2$, we have $c-a_2 \notin \mathcal X$.
		This implies $1=\dim ((0,c-a_2)) =\dim((a_2,c)) \leq \dim Z = \dim Y$ by Proposition \ref{prop:dim}(2) and Definition \ref{def:dimension}$(2)_1$.
		This contradicts the assumption that $\dim Y=0$.
		We have proven that $k_1 \leq \kappa$.
		Therefore, $Y$ is contained in $\mathcal C^{\kappa}(a_1)$.
		\medskip
		
		\textbf{Case (iii).}
		We can treat the case in which $Y$ is of the form $\mathcal C^{k_1}(a_1) \cap \mathcal D_{k_2}^-(a_2)$ and $\mathcal C^{k_1}(a_1) \cap \mathcal D_{k_2}^-(a_2) \neq \mathcal C^{k_1}(a_1)$ as in the previous case.
		We omit the proof.
		\medskip
		
		\textbf{Case (iv).}
		Let us consider the case in which $Y$ is of the form $\mathcal D_{+}^{k_1}(a_1) \cap \mathcal D_{-}^{k_2}(a_2)$.
		Observe that $a_1<a_2$.
		By symmetry, we may assume that $k_1 \geq k_2$.
		We have $a_1 \not\sim^{k_1}a_2$ because $Y \neq \emptyset$.
		We want to show $k_1 \leq \kappa$.
		
		Assume for contradiction $k_1>\kappa$.
		Suppose $k_2 \geq \kappa$.
		Then, $(2a_1+a_2)/3 \not\sim_{k_2} (a_1+2a_2)/3$ by Claim 2.
		This implies $(2a_1+a_2)/3 \not\sim_{\kappa} (a_1+2a_2)/3$.
		By equality (\ref{eq:3rd}), we have $1 = \dim (((2a_1+a_2)/3,(a_1+2a_2)/3))$.
		Since $Y$ contains the interval $((2a_1+a_2)/3,(a_1+2a_2)/3)$, we have $\dim Y=1$, which is absurd.
		
		Suppose $k_2<\kappa$.
		Put $Z:=Y \cup \mathcal C^{\kappa}(a_2)$.
		We have $\dim Z=\dim Y$ by Definition \ref{def:dimension}$(2)_1$ because $\dim C^{\kappa}(a_2)=0$.
		We have $a_1 \not\sim^{\kappa} a_2$ because $a_1 \not\sim^{k_1}a_2$ and $k_1>\kappa$
		By Claim 2, we have $(a_1+a_2)/2 \not\sim^{\kappa} a_2$.
		We get $\dim(((a_1+a_2)/2,a_2))=1$ by equality (\ref{eq:3rd}).
		Since the interval $((a_1+a_2)/2,a_2)$ is contained in $Z$, we have $1 \leq \dim Z = \dim Y$, which is absurd.
		We have proven $k_1 \leq \kappa$.
		
		Assume for contradiction that $a_1 \not\sim^{\kappa} a_2$.
		We have $(2a_1+a_2)/3 \not\sim_{\kappa} (a_1+2a_2)/3$ by Claim 2.
		In the same way as before, we can show $\dim Y=1$, which is absurd.
		We get $a_1 \sim^{\kappa} a_2$.
		$Y$ is contained in $\mathcal C^{\kappa}(a_1)$.
	\end{proof}
	
	We have completed the proof of Theorem \ref{thm:weakly}.
\end{proof}

\subsection{Dimension functions enjoying frontier property}
Finally, we show that there exists only one dimension function satisfying the extra property called the frontier property for definably complete locally o-minimal structures.

Let $\mathcal M$ be a first-order topological structure in the sense of \cite{Pillay2}.
A dimension function $\dim:\myDef(\mathcal M) \to \mathbb N \cup \{-\infty\}$ possesses the \textit{continuity property} if, for every definable set $X$ and definable function $f:X \to M$, the set $\mathcal D(f)$ of points at which $f$ is discontinuous is definable and $\dim \mathcal D(f)<\dim X$.

$\dim$ possesses the \textit{frontier property} if $\dim (\partial X)<\dim X$ for every definable set $X$, where $\partial X:=\mycl(X) \setminus X$ denotes the frontier of $X$.

\begin{proposition}\label{prop:equiv_cont_frontier}
	Let $\mathcal M=(M,\ldots)$ be a first-order topological structure.
	Let $\dim:\myDef(\mathcal M) \to \mathbb N \cup \{-\infty\}$ be a dimension function.
	If $\dim$ possesses the continuity property, then it enjoys the frontier property.
\end{proposition}
\begin{proof}
	Let $X$ be an arbitrary definable set.
	Pick two distinct elements $a$ and $b$ from $M$.
	Consider a definable function $f:\mycl(X) \to M$ defined by $f(x)=a$ if $x \in X$ and $f(x)=b$ in $x \in \partial X$.
	The set $\mathcal D(f)$ of points at which $f$ is discontinuous coincides with $\partial X$.
	This implies $\dim \partial X < \dim X$.
\end{proof}

\begin{proposition}\label{prop:local_omin}
	Let $\mathcal M$ be a definably complete locally o-minimal structure.
	Let $\dim:\myDef(\mathcal M) \to \mathbb N \cup \{-\infty\}$ be a dimension function possessing the frontier property,
	Then, $\dim$ coincides with the topological dimension function.
\end{proposition}
\begin{proof}
	Let $M$ be the universe of $\mathcal M$.
	Observe that $\dim X \leq 1$ for every definable subset $X$ of $M$ by Definition \ref{def:dimension}(1,2) because of the inclusion $X \subseteq M$.
	Let $I$ be an arbitrary open interval.
	If $I=M$, then $\dim I=1$ by Definition \ref{def:dimension}(1).
	If $I \neq M$, then the frontier $\partial I$ of $I$ is not empty.
	We have $1 \geq \dim I> \dim \partial I \geq 0$ by Definition \ref{def:dimension}(1) and the frontier property.
	We have obtained $\dim I=1$.
	
	Let $X$ be a definable subset of $M$.
	If $X$ has a nonempty interior, $X$ contains a nonempty open interval $I$. 
	We have $1 \geq \dim X \geq \dim I=1$ by  Definition \ref{def:dimension}(2).
	Suppose $X$ has an empty interior.
	Recall that $X$ is discrete and closed.
	We obtain $\dim X=0$ by Lemma \ref{lem:discrete_closed}.
	
	We have shown that $\dim$ coincides with the topological dimension function on $\myDef_1(\mathcal M)$.
	By Proposition \ref{prop:dim}(1), these two completely coincide with each other.
\end{proof}

\end{document}